\documentclass[11 pt,a4paper]{article}  

\usepackage{times}
\usepackage[pdftex]{hyperref}
\usepackage[latin1]{inputenc}
\usepackage{amsmath,indentfirst,qsymbols,graphicx,psfrag,amsfonts,stmaryrd,hyperref,color,enumerate}
\usepackage[numbers,sort&compress]{natbib}
\usepackage{paralist,vmargin}
\usepackage[labelfont=bf,font=it]{caption}
\usepackage{amsthm} 
\usepackage[normalem]{ulem}

\setpapersize{A4}
\setmarginsrb{2.5cm}{2cm}{2.5cm}{2cm}{.4cm}{4mm}{0cm}{7.5mm}

\hypersetup{
    unicode=false,          
    pdftoolbar=true,        
    pdfmenubar=true,        
    pdffitwindow=true,      
    pdfnewwindow=true,      
    colorlinks=true,        
    linkcolor=blue,         
    citecolor=blue,         
    filecolor=blue,         
    urlcolor=blue           
}

\def \1{\textbf{1}}

\def \sy{{\sf y}}

\def \se{{\sf e}}
\def \sB{{\sf B}}

\def \bw{{\bf w}}

\def \bT{{\bf T}}

\def \bar{\overline}
\def \ben{\begin{eqnarray}}
\def \een{\end{eqnarray}}
\def \ba{\begin{align}}
\def \ea{\end{align}}

\def \be{\begin{eqnarray*}}
\def \ee{\end{eqnarray*}}

\def \beq{\begin{equation}}
\def \eq{\end{equation}}
\def \bs{{\bf s}}
\def \build#1#2#3{\mathrel{\mathop{\kern 0pt#1}\limits_{#2}^{#3}}}

\def \ba{{\bf a}}

\def \bx{{\bf x}}

\def \cS{{\cal S}}
\def \cT{{\cal T}}

\def \dd{\xrightarrow[n]{(d)}}

\def \eref#1{(\ref{#1})}
\def \floor#1{\lfloor#1\rfloor}

\def \bF{{\bf F}}

\def \l{\left}

\def \proba{\xrightarrow[n]{(proba.)}}

\def \r{\right}
\def \sous#1#2{\mathrel{\mathop{\kern 0pt#1}\limits_{#2}}}
\def \sur#1#2{\mathrel{\mathop{\kern 0pt#1}\limits^{#2}}}

\newqsymbol{`B}{\mathcal{B}}
\newqsymbol{`E}{\mathbb{E}}
\newqsymbol{`I}{\mathbb{I}}
\newqsymbol{`N}{\mathbb{N}}
\newqsymbol{`O}{\Omega}
\newqsymbol{`P}{\mathbb{P}}
\newqsymbol{`Q}{\mathbb{Q}}
\newqsymbol{`R}{\mathbb{R}}
\newqsymbol{`W}{\mathbb{W}}
\newqsymbol{`Z}{\mathbb{Z}}
\newqsymbol{`a}{\alpha}
\newqsymbol{`e}{\varepsilon}
\newqsymbol{`o}{\omega}
\newqsymbol{`t}{\tau}
\newqsymbol{`w}{{\cal W}}
\def\cro#1{\llbracket#1\rrbracket}

\SetSymbolFont{stmry}{bold}{U}{stmry}{m}{n}

\def \Neigh{{\sf Neigh}}

\newcommand{\dist}{\operatorname{d}}

\newcommand{\sk}{\operatorname{Sk}}
\newcommand{\bbR}{\mathbb R}
\newcommand{\bbU}{\mathbb U}
\newcommand{\da}{\downarrow}
\newcommand{\bG}{\mathbf G}

\newcommand{\cB}{\mathcal B}

\newcommand{\btau}{\boldsymbol \tau}
\newcommand{\bgam}{\boldsymbol \gamma}
\newcommand{\black}{\color{black}}

\newcommand{\cP}{\mathcal P}
\newcommand{\E}[1]{\ensuremath{\mathbb{E} \left[#1 \right]}}

\newcommand{\Ec}[1]{\ensuremath{\mathbb{E} [#1]}}
\newcommand{\p}[1]{{\mathbb P}\left(#1\right)}
\newcommand{\pc}[1]{{\mathbb P}(#1)}

\newcommand{\bbD}{\mathbb D}
\newcommand{\bbC}{\mathbb C}

\newcommand{\sv}{{\sf v}}
\newcommand{\sw}{{\sf w}}
\newcommand{\cF}{\mathcal F}

\newcommand{\cN}{\mathcal N}

\newcommand{\cK}{\mathcal K}

\newcommand{\cZ}{\mathcal Z} 
\newcommand{\I}[1]{\ensuremath{\mathbf{1}_{ \{ #1 \} }}}
\def \CC{{\sf CC}}
\def \FCP{{\sf CP}}

\begin{document}

\newtheorem{fig}{\hspace{2cm} Figure}
\newtheorem{lem}{Lemma}
\newtheorem{defi}[lem]{Definition}
\newtheorem{pro}[lem]{Proposition}
\newtheorem{theo}[lem]{Theorem}
\newtheorem{cor}[lem]{Corollary}
\newtheorem{algo}{Algorithm}
\theoremstyle{definition}
\newtheorem{note}[lem]{Note}
\newtheorem{conj}[lem]{Conjecture}
\newtheorem{Ques}[lem]{Question}
\newtheorem{rem}[lem]{Remark}

\baselineskip=15pt

\title{\bf A new encoding of coalescent processes. \\ Applications
to the additive and multiplicative cases.}
\date{\today}
\author{ Nicolas Broutin \\ Inria Paris--Rocquencourt
        \and Jean-Fran\c{c}ois Marckert\\CNRS, LaBRI, Université Bordeaux}

\maketitle

\begin{abstract}
We revisit the discrete additive and multiplicative coalescents, starting with $n$ particles 
with unit mass. These cases are known to be related to some ``combinatorial coalescent processes'': 
a time reversal of a fragmentation of Cayley trees or a parking scheme in the additive case, and 
the random graph process $(G(n,p))_p$ in the multiplicative case. Time being fixed, encoding these 
combinatorial objects in real-valued processes indexed by the line is the key to describing
the asymptotic behaviour of the masses as $n\to +\infty$.

We propose to use the Prim order on the vertices instead of the classical breadth-first (or depth-first) 
traversal to encode the combinatorial coalescent processes. In the additive case, this yields 
interesting connections between the different representations of the process. 
In the multiplicative case, it allows one to answer to a stronger version of an open question of Aldous
[\emph{Ann.\ Probab.}, vol.\ 25, pp.\ 812--854, 1997]: 
we prove that not only the sequence of (rescaled) masses, seen as a process indexed by the time $\lambda$, 
converges in distribution to the reordered sequence of lengths of the excursions
above the current minimum of a Brownian motion with parabolic drift $(B_t+\lambda t - t^2/2, t\geq 0)$, 
but we also construct a version of the standard augmented multiplicative coalescent of Bhamidi, Budhiraja 
and Wang [\emph{Probab.\ Theory Rel.}, to appear] using an additional Poisson point process.

\end{abstract}

\section{Introduction}

Consider a family of weighted particles (carrying a mass, or a size) which (informally) merge according 
to the following rule: given some non-negative symmetric collision kernel $K$, each pair of particles with 
masses $x$ and $y$ collides at rate $K(x,y)$, upon which they coalesce to form a single new particle of mass 
$x+y$ (later on, this is sometimes referred to as a \emph{cluster}).
A mean-field model is provided by Smoluchowski's equations \cite{Smoluchowski1916}, which consist 
in an infinite system of ordinary differential equations characterising the joint evolution of 
the densities of particles of each mass as time goes. The systems are only solved in some special cases, 
among which one may cite the cases when the kernel is either additive, $K(x,y)=x+y$, or multiplicative, 
$K(x,y)=xy$ (\cite{Aldous97survey,Bertoin2006}, see \cite{FoLa} for more recent and general results).

Arguably, one of the objectives in the field of coalescent processes is to tend towards models of 
physical systems that would be more realistic ``at the particles level'', even if many of the features 
of real systems are still ignored, starting with the positions in space and energies of the particles. 
For an overview of the literature on these issues, and of the relation between coalescence processes 
and Smoluchowski's equations, we refer the interested reader to Aldous' survey \cite{Aldous97survey}, 
Pitman \cite{Pitman2006}, or Bertoin \cite{Bertoin2006}. 

When the number of particles is finite, it is rather easy to define rigorously a Markov process having 
the dynamics discussed in the first paragraph above. One possible construction is the so-called 
Marcus--Lushnikov \cite{Marcus1968,Lushnikov1978a} coalescence process.
Informally, consider the masses as vertices of a complete graph, and equip the edges between vertices $i$ 
and $j$ with random exponential clocks with parameter $K(x_i,x_j)$. When the clock between $i$ and $j$ rings, 
replace the masses $x_i$ and $x_j$ on the nodes $i$ and $j$ by $x_i+x_j$ and $0$, respectively, and 
update the parameters of the clocks involving $i$ and $j$ so that the rates remain given by the kernel.
When the number of masses is infinite, the definition of coalescence processes is much more involved.
These issues are discussed and solved for important classes of kernels in Evans \& Pitman \cite{EvPi1998} 
and in Fournier \& L\"ocherbach \cite{F-L} (see also references therein, and \cite{Aldous97survey}). 

In this paper we focus on the additive and multiplicative kernels. 
Together with Kingman's coalescent (for which $K(x,y)=1$), the associated coalescence processes 
are somehow the simplest, but also some of the most important.  
This is mostly because of their manifestations in fundamental discrete models
that we will call hereafter \it combinatorial coalescence processes\rm. 
These ``incarnations'' are forest-like or graph-like structures modelling the coalescence at a 
finer level which (partially) keep track of the history of the coalescing events 
\cite{EvPi1998,AlPi1998a,aldous1997,Bertoin2006,BER,CL,Pitman1999b, Pitman2006,BhBuWa2013a}. 

\medskip
\noindent\textsc{Invasion percolation and linear representations.}\ 
Most importantly for us, both the additive and the multiplicative coalescents 
started with unit-mass particles admit a graphical representation as (a time change of) the 
process of level sets of some weighted graph: there exists some (random) graph and random weights 
such that at each time $t$, the clusters are the connected 
components of the graph consisting of all edges of weight at most $t$. 
We call such processes \emph{percolation systems}. 
For the multiplicative coalescent, the graph is simply the complete graph weighted by 
independent and identically distributed (i.i.d.) random variables uniform on $[0,1]$; 
the additive case arises when taking the graph as a uniformly random labelled tree, and 
the weights to be i.i.d.\ uniform that are also independent of the tree. 
The idea underlying our work relies on \emph{invasion percolation} \cite{ChLeKoWi1982a,LeBo1980a} or equivalently 
on \emph{Prim's algorithm} \cite{Jarnik1930,Prim1957} to obtain 
an order on the vertices of such percolation systems, which we refer to as the \emph{Prim order} 
and that is \emph{consistent} with the coalescent in a sense that we make clear immediately. 
Given a connected graph whose edges are marked with non-negative and distinct weights,
and a starting node, say $v_1$, Prim's algorithm grows a connected component from $v_1$, 
each time adding the endpoint of the lightest edge leaving the current component (see Section~\ref{seq:POLR}). 
Prim order ``linearises'' the coalescent in a consistent way: at all times, 
the clusters are intervals of the Prim order, so that, in particular the clusters that coalesce 
are always ``adjacent'' (Proposition \ref{pro:Prim-perc}). Furthermore, it is remarkable that this definition of an alternate (random) order 
makes the consistence in time transparent and exact at the combinatorial level. We believe that that this new 
point of view should lead to further advances in the study of coalescence processes. 

Aside from this new unifying idea which is interesting on its own, our main contributions about the 
multiplicative and additive coalescent are the following:
 
\medskip
\noindent\textsc{Multiplicative coalescent.}\ 
We prove that the representation of the asymptotic cluster masses in terms of the 
excursion lengths of a functional of a Brownian motion with parabolic drift that 
\citet{aldous1997} proved valid for some fixed time (convergence of a marginal) can be extended 
to a convergence as a time-indexed process (convergence as a random function). 
This answers in particular Question 6.5.3 p.\ 851 of Aldous \cite{aldous1997}. 
The combinatorial coalescence process of interest is the 
percolation process on the complete graph, which is nothing else than the 
classical (Erd\H{o}s--R\'enyi) random graph process $(G(n,p), p\in[0,1])$ (see Section \ref{sub:rgp}) 
seen around $p=1/n + O(1/n^{-4/3})$. 
Furthermore, we also construct a version of the standard augmented multiplicative coalescent of 
\citet*{BhBuWa2013b} using only a Brownian motion and a Poisson point process.
This process has been constructed as the scaling limit of the sequence of cluster sizes and excesses of 
critical random graphs, see Section~\ref{sec:results} for more details.
\smallskip

\noindent\textsc{Additive coalescent.}\ 
Here the central combinatorial model is the percolation process on a uniformly random 
labelled tree, hereafter referred to as $\FCP_+^{\textsc{ap}}$
which has initially been built using random forests. 
The construction due to \citet{Pitman1999b} (see also references there for a complete and long history of the problem) 
leads to a continuous representation of the standard additive coalescent in terms of the time reversal 
of a fragmentation process (the logging process) of the Brownian continuum random tree (see \citet{AlPi1998a}). 
We introduce a slight modification of the parking model, that we refer to as $\FCP_+^{\textsc{cl}}$, constructed by
\citet{CL} as an approximation of the additive coalescence process. 
Our model $\FCP_+$ is equivalent to $\FCP_+^{\textsc{cl}}$ up to a random time change.  
Again  $\FCP_+$ is a one-dimensional model in which only consecutive blocks merge as time evolves. 
Our contribution here is to unify these results by showing that the model $\FCP_+$
can be used to encode $\FCP_+^{\textsc{ap}}$. 
Similarly to what is done in \cite{CL}, the blocks (resp.\ limiting blocks) have a representation in 
terms of the excursion lengths of some associated random walks (resp.\ functional of the normalised Brownian 
excursion) indexed by a two-dimensional domain (space and time). 
In this case, the limiting process is the standard additive coalescent and its construction 
using a Brownian excursion was already known (\cite{BER,CL}). 

\section{Main results about additive and multiplicative coalescents}\label{sec:results}

We present here the consequences of our work in terms of coalescence processes.
Write $\ell^p_\da$ for the set of non increasing  sequences of non-negative real numbers 
belonging to $\ell^p$ equipped with the standard $\ell^p$ norm, 
$\|x\|_p=\l(\sum_i |x_i|^p\r)^{1/p}$. 
As explained in \citet{EvPi1998}, $\ell^p_\da$ is a convenient space to describe coalescence processes.
Consider an element $\bx=(x_i, i\geq 1)$ of this space as a configuration, $x_i$ 
being the mass of particle $i$.  When two particles with masses $x_i$ and $x_j$ merge, 
their masses are removed from $\bx$, and replaced by one mass $x_i+x_j$ and another one with 
mass zero, inserted at the positions that ensure that the resulting configuration remains 
a non-increasing sequence of masses. 

The Marcus--Lushnikov (\cite{Marcus1968,Lushnikov1978a}, see also \cite{aldous1997}) definition of 
the finite-mass additive (resp.\ multiplicative) coalescent can be extended to sequences of masses 
in $\ell^1$ (resp.\ $\ell^2$) (see \cite{AlPi1998a} and \cite{aldous1997}). More precisely,
Aldous \cite[Proposition~5]{aldous1997} (resp.\ Evans \& Pitman \cite[Theorem~2]{EvPi1998}) 
proved that there exists a Feller Markov process taking values in $\ell^2_\da$ (resp.\ $\ell^1_\da$) 
which has the dynamics of the multiplicative (resp.\ additive) coalescent. 

Let $X^n$ be the additive coalescent process started at time $0$ in the state $(1/n,\dots,1/n,0,0,\dots)\in \ell^1_\da$, 
a configuration with $n$ particles each having mass $1/n$. Evans \& Pitman \cite{EvPi1998} (see also Aldous \& Pitman 
\cite[Proposition 2]{AlPi1998a}), proved that
\[\Big(X^n\Big(t+\frac 1 2 \log n\Big)\Big)_{-\infty < t <+\infty} \dd \l(X^{\infty}_+(t)\r)_{-\infty < t <+\infty}\]
for the Skorokhod topology on 
$\bbD((-\infty,+\infty),\ell^1_\da)$, the space of cadlag functions from $(-\infty,+\infty)$ taking values 
in $\ell^1_\da$, where the limiting process is also an additive coalescent, called 
the \emph{standard} additive coalescent (see also Section~\ref{sec:Pitman-Aldous}).

In the multiplicative case, Aldous \cite[Proposition 4]{aldous1997} states that  starting with a configuration with $n$ particles of mass $n^{-2/3}$, when the parameters of the exponential clocks between clusters are the product of their masses, then the sorted sequence of cluster sizes present at time $n^{1/3}+t$ (for a fixed $t$) converges in distribution in $\ell^2_\da$ to some sequence  $\bgam^\times(t)$ (described below). In Corollary 24, he shows that there exists a Markov process, called the \emph{standard} multiplicative coalescent, whose distribution at time $t$ coincides with $\bgam^\times(t)$,  and whose evolution is that of the multiplicative (Marcus--Lushnikov) coalescent. Nevertheless, with the construction he proposes, he is not able to prove that, as a process, $\bgam^\times$ is the \emph{standard} multiplicative coalescent.
 
The marginals of these standard coalescents both possess a representation using Brownian-like processes. 
Let $\se$ be a normalised Brownian excursion (with unit length) and let $\sB$ be a standard Brownian motion. Define 
\begin{align*}
\sy^{(\lambda)}_{\times}(x) &= \sB(x)+\lambda x - x^2/2,& x\geq 0, \lambda\in`R\\
\sy^{(\lambda)}_{+}(x)      &= \se(x)-\lambda x ,& x\in[0,1], \lambda>0
\end{align*}
and consider the operator $\Psi$ on the set of continuous functions $f:\Lambda\to `R$  defined by
\ben\label{eq:Psi}
\Psi f(x)= f(x)-\min\{f(y)~: y\leq x\}, ~~~~~x \in \Lambda
\een
where  $\Lambda=[0,+\infty)$ or $\Lambda=[0,1]$ is the domain of $f$. 
An interval $I=[a,b]$ is said to be an excursion of $f$ (resp.\ of $f$ above its minimum) if 
$f(a)=f(b)=0$ and $b=\inf\{t>a, f(t)=0\}$ (resp.\ if $f(a)=f(b)=\min\{f(t)~:t\leq b\}$).  
An important property of $\Psi$ is the following immediate lemma, which is illustrated in Fig.~\ref{fig:illus}.
\begin{lem}\label{lem:Psi-exc}
If $f(0)=0$ and $g=\Psi f$, then $I=[a,b]$ is an excursion of $f$ above its minimum if and only if $I$ 
is an excursion of $g$ above 0.
As a consequence, when these are well-defined, the multiset of the $k$ largest excursion sizes of 
$f$ above its minimum and of $g$ above 0 coincide.
\end{lem} 
\begin{figure}[h]
\centerline{\includegraphics[width=9 cm]{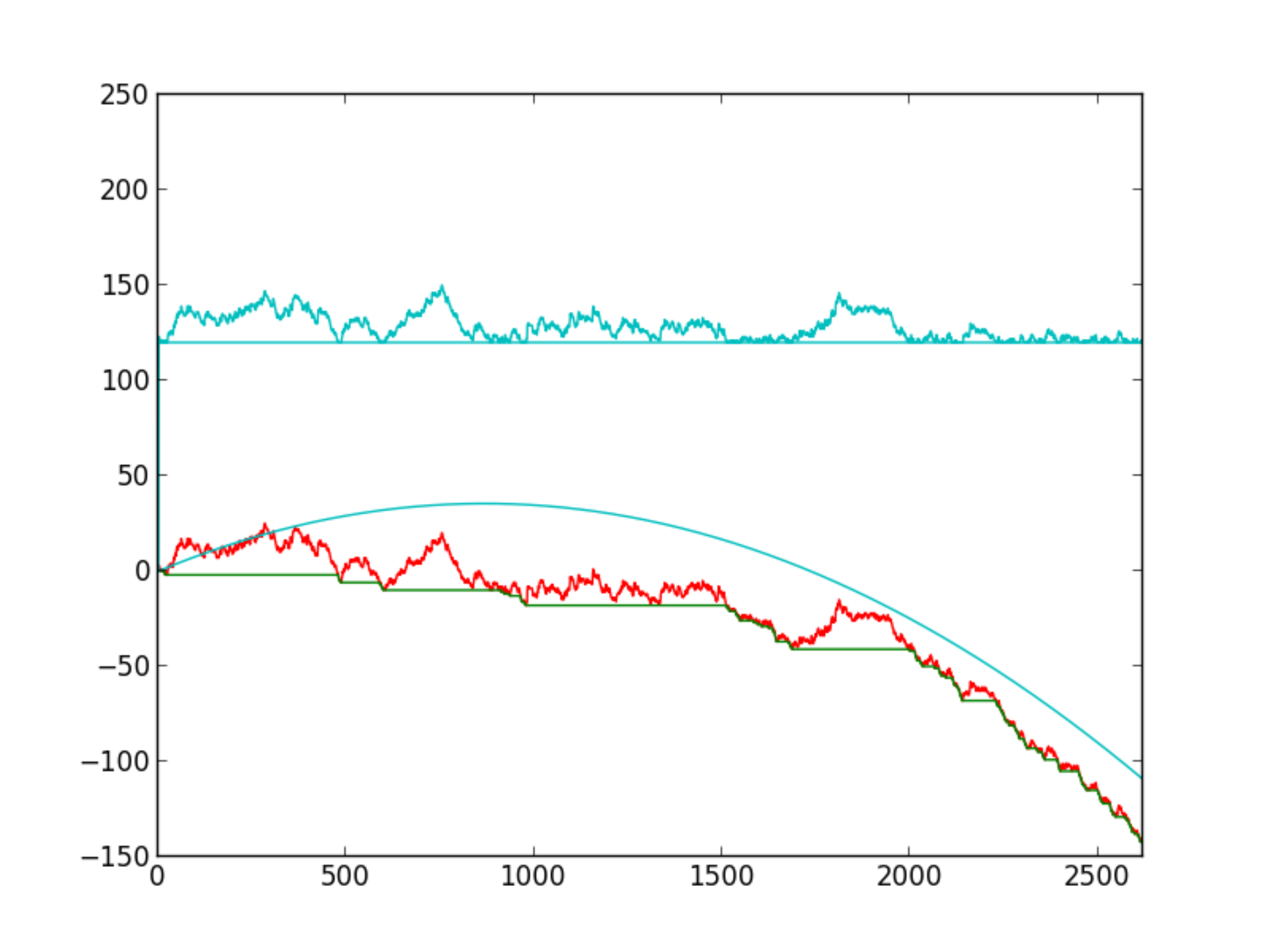}}
\caption{\label{fig:illus} 
A simulation of the processes $\sy_\times^{(\lambda),n}$ (bottom) and $\Psi \sy_\times^{(\lambda),n}$ (top), 
with the green line materializing the infimum process. Observe the correspondance between the excursions 
above 0 in the top picture with that above the minimum in the bottom one.}
\end{figure}

Let $\bgam^+(\lambda):=(\gamma^+_i(\lambda))_{i\ge 1}$ and 
$\bgam^\times(\lambda):=(\gamma^+_i(\lambda))_{i\ge 1}$ be the sequence of lengths of the 
excursions of $\Psi \sy_+^{(\lambda)}$ and of $\Psi \sy_\times^{(\lambda)}$, respectively,
sorted in decreasing order. Clearly, for any $\lambda\geq 0$,
$\gamma^+(\lambda)\in \ell_\da^1$ and, by Aldous \cite[Lemma 25]{aldous1997},  for any $\lambda\in `R$,
$\bgam^\times(\lambda)\in\ell_\da^2$.
Then, it is known that for any integer $k$ and real numbers $\lambda_1<\lambda_2<\dots< \lambda_k$ 
the vectors
\begin{align*}
(\bgam^+(-\ln(\lambda_1)), \dots, \bgam^+(-\ln(\lambda_k)))
\quad \text{and} \quad
(\bgam^\times(\lambda_1), \dots, \bgam^\times(\lambda_k))
\end{align*}
are distributed as the marginals at times $(\lambda_1,\dots, \lambda_k)$ of the standard additive and 
multiplicative coalescent, respectively (for the additive case, see Bertoin \cite{BER} and 
Chassaing--Louchard \cite{CL}; for the multiplicative case, see Aldous \cite{aldous1997} for 
the marginal convergence, and \citet{BhBuWa2013b} for the finite-dimensional distributions).

Bertoin \cite{BER} also proved that the process $(\bgam^+(-\ln(\lambda)))_{\lambda \geq 0}$ is a 
version of the standard additive coalescent.
A similar statement has been announced by Armendariz \cite{Armendariz2001} for 
$(\bgam^\times(\lambda))_{\lambda \in \bbR}$, but has never been published. 
Both \cite{BER} and \cite{Armendariz2001} argue directly in the 
continuum (\citet[Theorem 4.2]{CL} proceeded from a parking 
scheme, see Section \ref{sec:CL}, and proved only convergence of marginals). 
The main purpose of this paper is to give a simple and unified proof of these results based
on discrete versions of the coalescents.  
The objects involved are, as we said earlier, a parking scheme in the additive case and 
the random graph process $(G(n,p))_p$ in the multiplicative one. More precisely, our approach relies on 
encodings of these objects using discrete analogues of $\sy_+^{(\lambda)}$ and 
$\sy_\times^{(\lambda)}$, denoted by $\sy_+^{(\lambda),n}$ and $\sy_\times^{(\lambda),n}$. 
The associated processes 
\[\bgam^+(n,\lambda):=(\gamma^+_i(n,\lambda))_{i\ge 1}
\qquad \text{and} \qquad
\bgam^\times(n,\lambda):=(\gamma^+_i(n,\lambda))_{i\ge 1}\] will be seen (and this is standard) 
to coincide with lengths of their excursions above their respective minima (up to some details, see Note \ref{note:disc}). 

{Using the Prim order alluded above, the strength of these encodings will appear to be that 
the lengths of the excursions of $\Psi \sy_+^{(\lambda),n}$ (resp.\ $\Psi \sy_+^{(\lambda),n}$) 
correspond, up to a time change and a normalisation, to the  cluster sizes in an additive 
(resp.\ multiplicative) coalescent process, as a time-indexed process ($\lambda$ plays the role of time).
In particular, as $\lambda$ grows, only successive excursions of $\Psi \sy^{(\lambda),n}$ merge, 
which translates the fact that the Prim order \emph{linearises} the additive 
and multiplicative processes, in the sense that it makes them consistent 
with a linear order. 

Again, the construction in the additive case is close to that of 
Chassaing--Louchard \cite{CL} where the same property holds. As developed in 
Section \ref{sec:qsdq}, the novelty 
here is that our combinatorial additive coalescent corresponds to the 
linearisation of the time reversal of a 
fragmentation of a uniform Cayley tree defined by Pitman (see Section~\ref{sec:Pitman-Aldous}).
We show that in a suitable space
\[\sy^{(\lambda),n}_+(x)\dd\sy^{(\lambda)}_+(x)\]
as a process indexed by $(\lambda,x)$ (see Theorem \ref{theo:main+}).

The linearisation in the multiplicative case is new and allows us to prove the 
convergence of $\sy^{(\lambda),n}_\times(x)$ to  $\sy^{(\lambda)}_\times(x)$ as a process 
indexed by $(\lambda,x)$ (see Theorem \ref{theo:mainx}).  
}

Using the properties of $\Psi$ and of the operator ``extraction of excursion sizes'', 
we prove:
\begin{theo}\label{thm:conv-coal} 
We have 
\ben\label{eq:qsdq}
(\bgam^{+,n}(\lambda): \lambda \geq 0) \dd (\bgam^{+}(\lambda): \lambda\geq 0)
\een
and 
\[(\bgam^{\times,n}(\lambda): \lambda\in \bbR) \dd (\bgam^{\times}(\lambda): \lambda\in \bbR)\]
in the sense of Skorokhod convergence on $\bbD(\bbR, \ell^1_\da)$ and $\bbD(\bbR, \ell_\da^2)$, respectively.
\end{theo}

A a corollary, using a correspondence with coalescence (which in the additive case amounts to clarifying the 
time change) we establish that
\begin{cor}\label{thm:main_continuous} 
The processes $(\bgam^+(e^{-t}))_{t\in \bbR}\sur{=}{(d)}(X^{\infty}_+(t))_{t\in \bbR}$ and 
$(\bgam^\times(\lambda))_{\lambda\in \bbR}$
are versions of the additive and multiplicative coalescent, respectively.
\end{cor}
There, the statement means that $(\bgam^+(e^{-t}))_{t\in \bbR}$ is a Markov process 
taking values in $\ell^1_\da$ such that for every $t$, $\bgam^+(e^{-t})$ is distributed as follows 
\cite{AlPi1998a}. Consider a Brownian continuum random tree $\cT$ \cite{Aldous1991b} with mass measure 
$\mu$ and length measure $l$ on its skeleton $\sk(\cT)$. 
Consider a Poisson point process $\cP$ of intensity measure $l\otimes ds$ on $\sk(\cT)\times [0,\infty)$. 
At time $s$, splits ${\cal T}$ at the marks $u$ such that $(u,t)\in \cP$ and $t\leq s$, and denote by 
$\bF(s):=(\cF_1(s), \cF_2(s),\dots)$ the sequence of the $\mu$-masses of the connected components (subtrees)
obtained, sorted in decreasing order. Then, for 
every $s\in \bbR$, we have $\bF(s)\in \ell^1_\da$ and $\|\bF(s)\|_1=1$. With this setting, 
$(\bgam^+(s))_{s\in \bbR}$ and $(\bF(s))_{s\in \bbR}$ have the same distribution, 
a result which is originally due to \citet{BER}. 

In the multiplicative case, this means that $(\bgam^\times(\lambda))_{\lambda \in \bbR}$ is a 
Markov coalescent process taking values in $\ell^2_\da$ such that for every $\lambda\in \bbR$, 
the vector $\bgam^\times(\lambda)$ is distributed as the limit rescaled component sizes of the 
random graph $G(n,p_\lambda(n))$ for 
\ben\label{eq:plambda}
p_\lambda(n)=\frac{1}{n}+\frac{\lambda}{n^{4/3}}.
\een 
The existence of such a process, the standard multiplicative coalescent, has been proved by  
\citet[][Corollary 24]{aldous1997} by resorting to Kolmogorov's extension theorem. 
Here, we provide an explicit construction of the process from a single Brownian motion. 
The fact that the coalescing rates are multiplicative is a direct consequence of weak 
convergence used for the construction.
The proofs of Theorem~\ref{thm:conv-coal} and of Corollary~\ref{thm:main_continuous} 
are postponed until Section~\ref{sec:proofs_sequences}.

\medskip
In the multiplicative case, we also construct a version of the \emph{standard augmented multiplicative 
coalescent} of \citet{BhBuWa2013b} as a ``decorated'' process of $\bgam^\times$. 
For a connected graph, let the \emph{excess} be the minimum number 
of edges that one must remove in order to obtain a tree. Then, the augmented multiplicative 
coalescent is the scaling limit of the sizes and excesses of the connected components of $G(n,p_\lambda(n))$, that is of $(\bgam^{\times,n}(\lambda),\bs^n(\lambda))$ where $\bs^n(\lambda)=(s_i^n,i\geq 1)$ and $s_i^n$ is the excess of the $i$th largest connected component of $G(n,p_\lambda(n))$. 
The zero-set $\{x \geq 0: \Psi \sy_\times^{(\lambda)}(x)=0\}$ separates the half-line $\bbR^+$ 
into countably many open intervals $(I_i(\lambda))_{i\ge 1}$ whose lengths are precisely 
the components of the vector $\bgam^\times(\lambda)$. 
Let $\Xi$ be a Poisson point process with unit rate on $\bbR^+\times \bbR^+$. Then, for each 
$\lambda \in \bbR$ and for each $i\ge 1$, let $s_i(\lambda)$ denote the number of points of $\Xi$ 
falling under the graph of $\Psi y_\times^{(\lambda)}$ on the interval $I_i(\lambda)$, the 
interval corresponding to the $i$-th longest excursion of $\Psi y_\times^{(\lambda)}$:
\[s_i(\lambda):=\#\big\{(x,w)\in \Xi: x\in I_i(\lambda), w\le \Psi \sy_\times^{(\lambda)}(x)\big\}.\] 
Then write $\bs(\lambda)=(s_i(\lambda))_{i\ge 1}$.
The state space of interest is now $\bbU_\da$ defined by 
\[
\bbU_\da :=\bigg\{
(\bx,\bs)\in \ell^2_\da\times `N^\infty: \sum_{i\ge 1} x_i s_i <\infty  
\text{~and~} s_i=0 \text{~whenever~} x_i=0\bigg\}
\]
endowed with the metric 
\[
\dist_\bbU( (\bx,\bs), (\bx',\bs') ) 
:= \Bigg(\sum_{i\ge 1} |x_i-x_i'|^2 \Bigg)^{1/2} 
+ \sum_{i\ge 1} |x_i s_i- x_i' s_i'|~.
\]

\begin{theo}\label{thm:augmented_conv}
The following convergence
\[
((\bgam^{\times,n}(\lambda),\bs^n(\lambda)): \lambda \in \bbR)
\dd 
((\bgam^\times(\lambda), \bs(\lambda)): \lambda \in \bbR)
\]
holds in $\bbD(\bbR, \bbU_\da)$.
In particular, $(\bgam^\times(\lambda), \bs(\lambda))_{\lambda\in \bbR}$ is a version 
of the standard augmented multiplicative coalescent.
\end{theo}

Observe that the metric structure of the connected 
components obtained in \cite{AdBrGo2012a} from a similar representation at fixed $\lambda$ seems to be
ruined by the random Prim order. A careful look at Section~\ref{seq:POLR} should suffice to 
convince the reader that the very idea of obtaining a representation that is consistent in $\lambda$ is 
incompatible with tracking the internal structure of connected components.

\section{Combinatorial coalescence processes and their encodings}  

\subsection{The multiplicative case: critical random graphs}
\label{sub:rgp}

The aim of this part is to present some elements concerning the multiplicative coalescence 
processes, our new approach, and the main steps to the proofs of Theorem \ref{thm:conv-coal} 
and \ref{thm:main_continuous}.
    
We first define the random graph process on the vertex set 
$[n]:=\{1,2,\dots, n\}$, for a positive integer $n$. Let $E^n=\{\{i,j\}, i\neq j, i,j \in [n]\}$ 
denote the set of pairs of elements of $[n]$, the set of edges. Let $(U_e)_{e\in E^n}$ be a collection 
of i.i.d.\  uniform random variables on $[0,1]$. Let $G(n,p)$ be the graph on $[n]$ consisting of the edges $e\in E^n$
for which $U_e\le p$. Then, $(G(n,p))_{p\in[0,1]}$ is the classical random graph process \cite{Bollobas2001,JaLuRu2000}. 
It is a Markov process but not time-homogeneous (as it would have been if instead of uniform random variables 
we would have used exponential ones). The ordered sequence of 
sizes of connected components $(|C^n_i(t)|)_{i\ge 1}$ is also a Markov process, for which 
the initial state is $(1,1,\dots,1)$ and the components of the vector coalesce at rate which 
is proportional to the product of their values. Indeed, conditionally on $G(n,t)$, the 
next edge to be added is equally likely among the ones which are not already present, 
so that the probability that it joins a vertex of $C_i^n(t)$ to one of $C_j^n(t)$ is proportional 
to $|C_i^n(t)|\times |C_j^n(t)|$.  Thus, up to a time change, the connected components in $G(n,p)$ 
behave as the multiplicative coalescent. 
 
To obtain a limit theorem for these connected component sizes as a time-indexed process, our approach 
uses ideas from the proof by Aldous \cite{aldous1997} of the convergence at a fixed time. He encodes the connected
components into a discrete random real-valued process whose convergence implies the convergence of
the sizes of the connected component. To get suitable limit theorem, the probability $p$ has to be chosen 
inside the \emph{critical window}, that is of the form $p=p_\lambda(n)$, as defined in \eref{eq:plambda}.
The method of Aldous relies on a breadth-first traversal of the graph $G(n,p_\lambda(n))$.
It is easily seen that, in the context of the random graph $G(n,p)$, the following ``smallest-label-first'' 
traversal has the same distribution, so that the results of Aldous \cite{aldous1997} apply when using this modified 
algorithm. 
In the following, we call \emph{neighbourhood} of a set of vertices $S$ the collection of nodes that 
have an edge to a node in $S$, but are not themselves in $S$.

\begin{algo}[Standard traversal]\label{alg:aldous}
Traverse the vertices of a graph on $[n]$ as follows:
\begin{compactitem}[\textbullet]
    \item Start at step $k=1$ with node $v_1=1$ and set $S_1=\{v_1\}$.
    \item At step $k+1\in \{2,\dots, n\}$, the nodes $v_1,\dots, v_k$ are already known,
    and we have $S_k=\{v_1,\dots, v_k\}$. 
    Let $v_{k+1}$ be the node with smallest label
    among the neighbours of $S_k$, or if the neighbourhood of $S_k$ is empty, $v_{k+1}$ is 
the node  with smallest label in $[n]\setminus S_k$.
\end{compactitem}  
\end{algo}

Denote by $Z_k^{n,p_\lambda(n)}$ the size of the neighbourhood of $S_k$ and set 
\[Y_k^{n,p_\lambda(n)}= Z_k^{n,p_\lambda(n)} -\#\{j \leq k, Z_k^{n,p_\lambda(n)}=0\}.\]
Then, the sizes of the connected components of $G(n,p_\lambda(n))$ are precisely the 
lengths of the intervals between the zeros of $(Z_k^{n,p_\lambda(n)}, 1\le k\le n)$ 
(see Section~1.3 of \cite{aldous1997} and Lemma~\ref{lem:rel}).
Then define 
\[y^{n,(\lambda)}(x):=\frac{Y^{n,p_\lambda(n)}(n^{2/3}x)}{n^{1/3}}
\qquad \text{and}\qquad 
z^{n,(\lambda)}(x):=\frac{Z^{n,p_\lambda(n)}(n^{2/3}x)}{n^{1/3}}.\]
\citet{aldous1997} proved that, for any fixed $\lambda\in \bbR$,
\ben\label{eq:mono-lambda} y^{n,(\lambda)} 
\dd \sy^{(\lambda)}_\times
\qquad \text{and}\qquad
z^{n,(\lambda)}(x)\dd \Psi\sy^{(\lambda)}_\times, 
\een
where the convergence holds for the topology of uniform convergence on every compact.  

We propose to modify a bit the traversal of the graph in Algorithm~\ref{alg:aldous}: instead of using the labels 
order to define the traversal, use the Prim order (see Section \ref{seq:POLR} for more details): that is proceed as in 
Algorithm~\ref{alg:aldous} but replace the two instances of ``the node  with smallest label'' by ``the node with 
smallest Prim rank''. 
Observe that the Prim order on $G(n,p)$ is defined using the weights $(U_e)_{e\in E^n}$ only, and thus does not depend 
on $p$, unlike the order given by the standard traversal used by Aldous. In the following, we add the subscript ``$\times$''
in the notation for the random variables defined using this modified Prim traversal, and we set 
\[\sy^{n,(\lambda)}_\times(x):=\frac{Y^{n,p_\lambda(n)}_\times(n^{2/3}x)}{n^{1/3}}\]
where $Y_\times$ is assumed to be interpolated between integer
points. Observe that in the superscript of $\sy^{n,(\lambda)}$, the
superscript $(\lambda)$ corresponds to the parameter $p_\lambda(n)$ defined in \eqref{eq:plambda}.
In the following, the processes $\lambda\mapsto \sy^{n,(\lambda)}_\times$ and $\lambda\mapsto \sy_\times^{(\lambda)}$ 
are denoted more simply by $\sy_\times^{n}$ and $\sy_\times$.
\begin{theo}\label{theo:mainx} The following convergence holds in $\bbD(`R, \bbC([0,\infty),`R))$,
\ben 
\label{eq:mast1}\sy^{n }_\times \dd \sy_\times 
\een 
where $\bbC([0,\infty),`R)$ is the set of continuous functions 
from $[0,\infty)$ with values in $`R$.
\end{theo}
The proof is postponed until Section~\ref{sec:enc_mult} (and more details on the distribution of 
$(\sy^{n,(\lambda)}_\times, \lambda \in `R)$ are given in Section \ref{sec:dissyn}).

Observe that for a fixed $\lambda$, the convergence \eref{eq:mono-lambda} obtained by \citet{aldous1997} 
implies that $\sy^{n,(\lambda)}_\times\to \sy_\times^{(\lambda)}$ in distribution, provided that we additionally 
prove that $\sy^{n,(\lambda)}_\times$ and  $y^{n,(\lambda)}$ have the same distribution, a fact that we prove in
Lemma~\ref{lem:red}. We also provide a direct proof of the fixed-time convergence in 
Section~\ref{sec:newproof_Aldous}.

\begin{note}\label{note:disc}
When we are talking about interpolated discrete processes and discrete coalescence, a slight 
modification in the definition of excursions has to be done in order to obtain an exact correspondence 
between the cluster sizes and excursion sizes. 
For the excursion away from zero, $f(a)=f(b)=0$ and $b=\inf\{t>a: f(t)=0\}$ has to be replaced by 
$f(a)=f(b)=0$ and $b=\inf\{t>a+\alpha_n: f(t)=0\}$, where $\alpha_n$ is the size of a rescaled discrete step. 
The discrete excursions above the current minimum are defined by 
$a=\min\{t: f(t)=f(a)\}$, and $b=\min\{t: f(t)=f(b)\}$ with $f(b)=f(a)-\beta_n$,
where $\beta_n$ is the space normalisation. 
\end{note}

\begin{note}\label{note:homogeneous}
In order to obtain exactly the (time-homogenenous) Markovian coalescent from the random graph process, one only 
needs to consider a new time parameter given by $t=-\ln(1-p_\lambda(n))$. However,  as $n\to\infty$, 
$-\ln(1-p_\lambda(n))$ and $p_\lambda(n)$ behave similarly (at the second order), and the study of coalescent 
can be done using $p_\lambda(n)$.  We use $p_\lambda(n)$ in order to stay closer to the random graph model, 
as did Aldous \cite{aldous1997}.
\end{note} 

\subsection{Additive coalescence processes}

In the three next subsections, we treat the different combinatorial coalescence processes 
related to the additive coalescent. The main references here are \cite{EvPi1998, Pitman1999b, AlPi1998a, BER, CL,Pitman2006}. 

\subsubsection{The combinatorial coalescence process $\FCP_+^{\textsc{ap}}$}
\label{sec:Pitman-Aldous}

The following discussion relies on the results by \citet{AlPi1998a}, see also \citet[][(ii)' p.\ 170]{Pitman1999b}. 
We define a process of random forests of unrooted labelled 
trees $F(n,s)$, $s\ge 0$ as follows. At time $s=0$, the forest $F(n,0)$ consists of $n$ isolated trees 
$t_1,\dots,t_n$ where $t_i$ is reduced to the node $i$ alone. When the number of trees is $m$, wait 
an exponential random variable with parameter $m-1$, then pick a pair of trees $(t_i,t_j)$ with 
$1\leq i<j\leq m$ with probability $(|t_i|+|t_j|)/(n(m-1))$, and add an edge between a uniform node 
in $t_i$ and a uniform node in $t_j$.
Considering only the rescaled tree sizes $C^{n,s}=(C_i^{n,s}/n,i\geq 1)$ of the forest $F(n,s)$ 
(sorted and completed by an infinite sequence of 0), we have
\[(C^{n,s},s\geq 0)\sur{=}{(d)}(X_+^n(s), s\ge 0).\]
Since any pair of trees coalesces with probability proportional to the sum of their sizes, we 
just need to check that the same time-scale arises in the additive coalescent. This is indeed the case, 
since in the latter, when $m$ particles with total unit mass are present, the first coalescence occurs 
after a time equal to the minimum of independent exponential random variable with parameters $K(x_i,x_j)$, 
and $\sum_{1\leq i<j \leq m}K(x_i,x_j)=\sum_{1\leq i<j \leq m}x_i+x_j =m-1.$ Thus in the present coalescent, 
if one takes $(E_i,1\leq i\leq n-1)$ a sequence of i.i.d. exponential r.v. with parameter 1, then the 
number of coalescences before time $s$ is
\[M^n(t)=\max\bigg\{m\ge 0: \sum_{j=1}^m \frac{E_j}{n-j}\leq s,\bigg\}\]
and
\[n^{-1/2}\l(n-M^n\l(s+\frac{1}2\log n\r)\r)\to e^{-s},\]
where the convergence holds in $\bbD((-\infty,+\infty),\bbR)$ (the convergence holds in fact 
uniformly on any compact $[-\lambda_\star,\lambda^\star]$).

\subsubsection{The combinatorial coalescence process $\FCP_+^{\textsc{cl}}$}
\label{sec:CL}

We now present quickly the model and results of Chassaing \& Louchard \cite{CL}.
Assume $n$ cars park on a circular parking, identified with $\mathbb{Z}/n\mathbb{Z}$, 
according to the following algorithm. Let $({\sf Ch}_i,1\leq i\leq n)$ be a family of 
i.i.d.\ random variables uniform on $\mathbb{Z}/n\mathbb{Z}$. The cars park successively.
When the $i-1$ first cars have already parked, car $i$ chooses place ${\sf Ch}_i$ 
and parks at the first available place in the list ${\sf Ch}_i$, ${\sf Ch}_{i}+1 \!\!\!\mod\! n$, 
${\sf Ch}_{i}+2 \!\!\!\mod\! n$... 
Assume that $m$ cars are parked and call \it  block \rm a sequence of adjacent occupied 
places. As explained in \cite[Section 8]{CL} to get a suitable relation with the additive coalescent, 
the correct notion of size for a block is the number of cars consecutively parked plus one. 
This model coincides exactly with the Marcus--Lushnikov additive process up to a random time change.  
When $m=m(n)=\floor{n-\lambda\sqrt{n}}$ cars are parked, the large $n$ asymptotic 
evolution of the sizes of these blocks (sorted in decreasing order)
\[B^{n,\lambda}:=\frac{1}{n}(B^{n,\lambda}_i,i\geq 0)\] 
is given by the standard additive coalescent up to a time change. 
Here are some precisions on this time change: the time $\floor{n-\lambda\sqrt{n}}$ coincides with the number 
of coalescence done. From what we said above in the additive coalescent, this occurs at a random time of order    $t+(1/2)\log (n)$ for $t$ such that $\exp(-t)=\lambda$, so that for a fixed $\lambda >0$  one can prove
\[B^{n,\lambda}\dd X^{\infty}(-\ln(\lambda)).\]
\black
More precisely, Chassaing and Louchard obtained in \cite[Theorem 1.3]{CL} the convergence of the sizes 
of the $k$ largest blocks to that of the $k$ largest excursions of $\Psi \sy_+^{(\lambda)}$. 

\subsubsection{The new combinatorial coalescence process $\FCP_+$}
\label{sec:our_model}

Our new combinatorial coalescence process is in the mean time very close to that of \citet{CL} and to that 
of Pitman (Section \ref{sec:Pitman-Aldous} above). The new idea of Prim's order makes the connection 
between these two models very clear, in
a way that is both different from the one discussed in \cite[Section 8]{CL}, and 
similar to our approach to the multiplicative coalescent.

Although the intuition comes from the percolation model on the uniformly random labelled tree, 
it is convenient for the proofs to construct the process $\FCP_+$ as follows. The connections 
with the parking model and the fragmentation on trees are made later on in Section~\ref{sec:add_unify}
Consider a sequence of i.i.d.\ Poisson random variables $(X(i),1\leq i \leq n)$
with parameter one, and associate to this sequence the random walk
\[Y(m)=\sum_{j=1}^m (X(j)-1),\qquad 0\leq m \leq n.\]
Denote by $\tau_{-1}=\inf\{ m~:~Y(m)=-1\}$ the hitting time of $-1$.

Further, denote by $(Y_+^n(m),0\leq m \leq n)$ the random walk $Y$ conditioned on $\tau_{-1}=n$. 
Denote by $(X_+^{n}(i),1\leq i \leq n)$ the increments of $Y_+^n$ 
(that is under the condition that $\tau_{-1}=n$). 
Now introduce an array $(U_{k}(\ell),1\leq k \leq n, 1\leq \ell \leq n)$ of i.i.d.\ ${\sf uniform}[0,1]$ 
random variables and, conditionally on the family $(X^{n}_+(i),1\leq i \leq n)$, 
define the family $(X^{n,(t)}_+(i), 0\leq t \leq 1, 1\leq i \leq n)$, by
\[X_+^{n,(t)}(i)=\sum_{j=1}^{X_+^{n}(i)} 1_{U_{i}(\ell)\leq t},
~~\textrm{ for any } t\in[0,1], 1\leq i \leq n.\] 
Hence, $X^{n,(0)}_+(i)=0$, $X^{n,(1)}_+(i)=X^{n}_+(i)$ and for any $i$,  $t\mapsto X_+^{n,(t)}(i)$ 
is non-decreasing. 
Define
\[Y^{n,(t)}_+(m)= \sum_{j=1}^m (X^{n,(t)}_+(j)-1),~~0\leq m \leq n, ~t\in[0,1].\]

From now on, consider that $(Y^{n,(t)}_+(m),0\leq m \leq n)$ is a continuous process in the variable
$m$, obtained by linear interpolation between integer points. Set  
\[\sy_+^{n,(\lambda)}(x)=\frac{Y^{n,(1-\lambda/\sqrt{n})}_+(nx)}{\sqrt{n}}, 
~~~ x\in[0,1], \lambda \geq 0.\]
The processes $(\lambda,x)\mapsto \sy^{n,(\lambda)}_+(x)$ and $(\lambda,x)\mapsto \sy_+^{(\lambda)}(x)$ 
are denoted more simply $\sy_+^{n}$ and $\sy_+$ in the sequel. 
They are seen as random variables taking their values in $\bbD(`R^+,\bbC([0,1],`R))$. 
In words, for fixed 
$\lambda$, $\sy_+^{(\lambda)}$ is a r.v.\ in $\bbC([0,1],`R)$. 
Seen as a process in $\lambda$ it is right-continuous with left limits.
From what we said earlier $\lambda\mapsto\sy_+^{n,(\lambda)}$ is non-increasing in $\lambda$.

For $\lambda=0$, $Y_+^{n,(0)}$ is just a random walk conditioned to  hit $-1$ at time $n$. 
By a generalisation of Donsker's invariance principle \cite{Donsker1952}, see for 
instance \cite{Kaigh1976,Marckert2007}, we have
\begin{equation}\label{eq:yxn}
\sy^{n,(0)}_+\dd \sy^{(0)}_{+},
\end{equation}
in $\bbC([0,1],\bbR)$ equipped with the topology of uniform convergence (recall that $\sy^{(0)}_{+}=\se$). 
The proof of the next theorem is postponed until Section~\ref{sec:CVFDDyplus}. 

\begin{theo} \label{theo:main+}
The following convergence holds in $\bbD(`R^+, \bbC([0,1],`R))$,
\[ \sy^{n}_+ \dd  \sy_+.  \]
\end{theo}

\subsubsection{Between a parking scheme and fragmentation of Cayley trees}
\label{sec:qsdq}\label{sec:add_unify}

\noindent\textsc{The parking scheme point of view on $\FCP_+$.}\ 
The construction in Section~\ref{sec:our_model} may be interpreted as follows in terms of a parking scheme: 
$X(i)$ is the number of cars whose first choice is place $i$ and that park at the first empty place to the right of $i$. 
The condition $\tau_{-1}=n$ amounts to saying that, in then end, the place $n$ is still empty 
(see \cite{CL} for more details). 
The random variable $X_+^{n,(t)}(i)$ represents the number of cars that have chosen place $i$ by time $t$. 
Observe that conditionally on $\tau_{-1}=n$, $\sum_{i=1}^n X(i)=n-1$, so that $\sum_{i=1}^n X^{n,(t)}_+(i)$ 
is binomial with parameters $n-1$ and $t$. In particular, this is random, unlike in \cite{CL}.
Hence, at time $t=1-{\lambda}/{\sqrt{n}}$ the number of coalescences that already occurred, denoted 
by $N^{n,\lambda}$, is binomial$(n-1,1-\lambda/{\sqrt{n}})$ and it follows that
\[W_{n,\lambda}:=\frac{n-N^{n,\lambda}}{\sqrt{n}}\dd \lambda,\]
the convergence holding in distribution in $\bbD([0,\lambda_\star], \bbR)$ for any $\lambda_\star$, 
since the convergence is uniform on any compact. Indeed one may check that, as a process on 
$[0,\lambda_\star]$, we have 
$W_{n,\lambda}=n^{-1/2}\sum_{i=1}^{n-1} 1_{U^{(i)}\leq \lambda/\sqrt{n}}$ for some $(U^{(i)}, i\ge 1)$ 
i.i.d.\ uniform on $[0,1]$. 
The process $\lambda \mapsto W_{n,\lambda}$ is non-decreasing, and its finite-dimensional 
distributions converge to those of the deterministic process $(\lambda, \lambda \geq 0)$ 
(convergence of the mean, and the variance goes to 0), and thus the convergence is almost sure (a.s.) 
on any compact (see \cite[Appendix]{Marckert2007} if more details are needed).

The convergence of this time change between our model and the discrete coalescence process, together 
with the convergence of the excursion sizes (as a process in $\lambda$) are the main tool to obtain 
the convergence to the additive coalescent. 

\medskip
\noindent\textsc{The percolation point of view.}\
Consider a uniform Cayley tree with $n$ vertices (uniformly labelled tree on $[n]$), and root it at the 
vertex labelled $1$. For a node in $[n]$, let its out-degree be the number of its neighbours that 
are further from the root. Then, it is folklore that the sequence of node out-degrees $(d_i,1\leq 1\leq n)$ 
where the nodes are sorted according to the breadth-first order
is distributed as $(X(i),1\leq i \leq n)$ conditional to $\tau_{-1}=n$, 
as described above. 
The random walk $Y^{n,(0)}_+$ appears in the literature as the \L ukasiewicz walk associated 
with a uniform Cayley tree \cite[see, e.g.,][]{Legall2005}.
Now, equip the edges of the Cayley tree with i.i.d.\ uniform weights $(U_e,e \in E)$ 
(independently of the tree) and keep the edges with weight smaller than $t$, discarding the others. 
One then obtains a forest. 
In this forest $F_t$, let $d_i(t)$ denote the out-degree of the node that had previously rank $i$ 
in the Cayley tree (so that $d_i(1)=d_i$). Then clearly, 
\ben
(d_i(t),~1\leq i \leq n)_{t\in[0,1]}\sur{=}{(d)}(X^{n,(t)}_+(i),~1\leq i \leq n)_{t\in[0,1]}.
\een
The following proposition is a consequence of Lemma \ref{lem:prpr} (and Lemma \ref{lem:red}):
\begin{pro}\label{pro:mach} 
Let $T$ be a uniform Cayley tree on $n$ vertices whose edges are 
(independently) equipped with i.i.d.\ uniform $[0,1]$ weights.
Let $(u_i,1\leq i \leq n)$ be the nodes sorted according to the Prim order (when the root 
node is $u_1=1$), and let $X^t(i)$ be the number of edges between $u_i$ and its children, 
that have a weight at most $t$. Then 
\[(X^t(1),\dots,X^t(n))_{t\in[0,1]}\sur{=}{(d)} (X_+^{n,(t)}(1),\dots,X_+^{n,(t)})_{t\in[0,1]}.\]
As a consequence the collection of excursion sizes of $\Psi Y^{n,(t)}_+$ 
at time $t$ evolves, up to a time change, as the additive coalescent. 
\end{pro} 

It is classical that a tree (or a forest) can be encoded by a  \L ukasiewicz walk. This walk encodes 
the sequence of node degrees  $(Y^{n,t}(j)=\sum_{k=1}^j (d_i^{(t)}-1))$. 
As a consequence, the sequence of sizes of the trees in $F_t$, sorted using the Prim order
correspond to the sequence of excursion sizes 
of $\Psi Y^{n,(t)}_+$, and this property is true as a process indexed by $t$.
This makes a connection between the results by \citet{AlPi1998a} and our representation of additive 
coalescent, and explains again the fact that the additive coalescent can be linearised.

\section{Prim's order and linear representations of coalescents}
\label{seq:POLR}

This part presents the main new idea underlying this work. 

\subsection{Prim's algorithm and coalescents}

In this section, we assume that an integer $n\geq 2$ is fixed.
Let $G=([n],E)$ be any connected graph, where the edges are marked by some weights,
$\bw=(w_e,e \in E)\in [0,1]^{\#E}$, some non-negative real numbers. 
The pair $(G,\bw)$ is said to be \emph{properly weighted} if the weights are distinct and 
positive. 

Prim's algorithm (or Prim--Jarn\'ik algorithm) is an algorithm which associates with 
any properly weighted graph $(G,\bw)$ its unique minimum spanning tree, the 
connected subgraph of $G$ that minimises the sum of the weights of its edges. 
It also defines a total order $\prec$ on the 
set of vertices.  Let us describe the nodes $u_1,\dots,u_n$ 
satisfying $u_1\prec u_2\prec \dots \prec u_n$. We will use below the notation $V_{i}$ 
for the set $\{u_1,\dots,u_i\}$.
 
First set $u_1=1$ and $V_1=\{u_1\}$. Assume that for some 
$1\leq i \leq n-1$, the nodes $u_1,\dots,u_i$ have been defined. Consider the set of weights 
$\{w_{\{a,b\}} ~|~a\in V_i, b\notin V_i\}$ of edges between a vertex of $V_i$ and another 
outside of $V_i$. Since all weights are distinct, the minimum is reached at a single pair 
$(a^\star,b^\star)\in V_i \times \complement V_i$. Set $u_{i+1}=b^\star$. 
This iterative procedure completely determines the \emph{Prim order} $\prec$. 
If one sets additionally $\pi_{i+1}=a^\star$, 
a classical result (not used in the paper) is that the minimum spanning tree is the 
tree on $[n]$ with set of edges $\{(\pi_i, u_i): 2\le i\le n\}$.
\begin{defi} We say that a set of nodes $\{v_1,\dots,v_\ell\}$ forms a \emph{Prim interval}, 
if $\{v_1,\dots,v_\ell\}=\{u_{i}, i \in \cro{a,a+\ell-1}\}$ for some $a$, that is if their 
\emph{Prim ranks} are consecutive. Any Prim interval can be written as $V_j\setminus V_i$ 
for some pair $(i,j)$. 
\end{defi} 

Given a properly weighted graph $(G,\bw)$, for any $t \in [0,1]$, $E_t(\bw)=\{e\in E: w_e\le t\}$ 
and $G_t(\bw)=([n],E_t)$ the graph whose edges are the edges of $E$ with weight at most $t$. 
The next proposition, which seems to be folklore in graph theory, is of prime importance to us. 
In the sequel we write $E_t$ and  $G_t$ for short, the weights being clear from the context.

\begin{pro}\label{pro:Prim-perc} Let $(G,\bw)$ be a properly weighted graph. For any $t\in[0,1]$, 
all the connected components of $G_t$ are Prim intervals. As a consequence, the coalescence of 
connected components arising when $t$ increases corresponds to coalescence of 
consecutive Prim intervals.
\end{pro}

\begin{figure}[hb]
\centerline{\includegraphics[height=3.5cm]{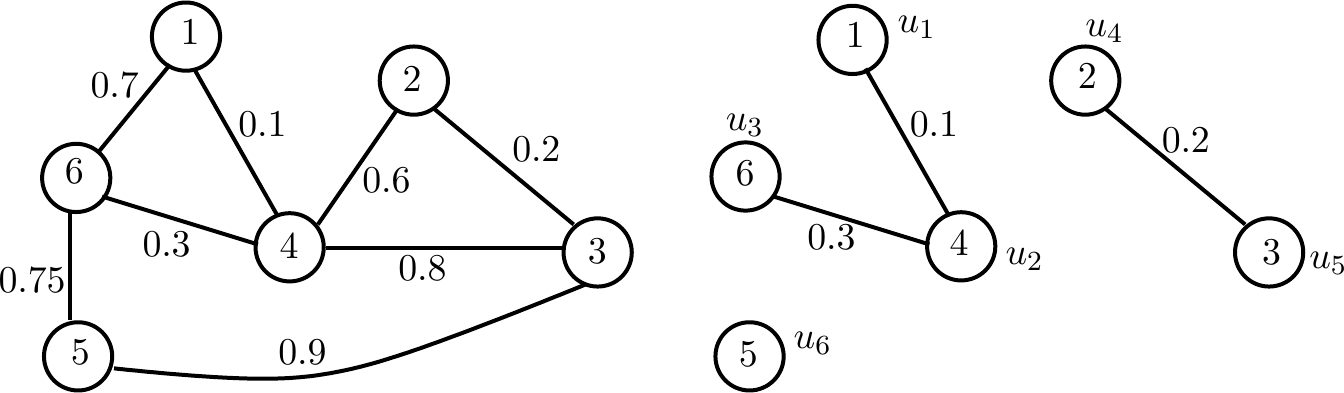}}
\caption{A realisation of the weighted complete graph $G$ on $[6]=\{1,2,\dots, 6\}$ where only the edges with 
weight at most $0.9$ are drawn. On the right, only the edges $e$ such that 
$w_e\leq 0.5$ are kept (this is $G_{0.5}$), and the $u_i$'s are the nodes sorted according 
to the Prim order. Notice that the connected components are intervals in the Prim order. 
In this example, $\cZ^{0.5}(1)=\{u_2\}, \cZ^{0.5}(2)=\{u_3\}, \cZ^{0.5}(3)=\varnothing, \cZ^{0.5}(4)=\{u_5\}, 
\cZ^{0.5}(5)=\varnothing, \cZ^{0.5}(6)=\varnothing$.}
\end{figure}

\begin{proof} Only the first statement needs to be proved. 
The graph $G_t$ is non-decreasing for $t\in[0,1]$, and since by hypothesis the weights are 
distinct and non-zero, we have $E_0=\varnothing$ and $E_1=E$. The finite set of weights/times 
$\{w_e, e\in E\}$ 
are the jumping times for the function $(G_t, 0\leq t \leq 1)$, and exactly $n-1$ of these 
dates $t_1,\dots,t_{n-1}$ modify the number of connected components. So the result 
needs only be checked at these times. 
For $t=t_0=0$ the result holds. Assume that, at time $t_k$ for some $0\leq  k \leq n-2$, 
all the connected components are consecutive intervals $(I_1,\dots,I_\ell)$ 
with $I_j=[a_j,b_j]$ and $a_{j+1}=b_j+1$. Denote by $e$ the edge which is added at time $t_{k+1}$. 
By hypothesis adding it decreases the number of connected components, and so its end points lie 
in two distinct intervals $I_{x}=[a_x,b_x]$ and $I_{y}=[a_y,b_y]$ for some $x<y$. 
If $y=x+1$ we are done, so assume for a contradiction that $y>x+1$. 
The weight $w_e$ is smaller than all those of the missing edges at time $t_{k+1}$; in particular, it is 
smaller than all the weights of the edges between $\cup_{i\leq x} I_i$ and $I_{x+1}$. 
But this is impossible, since it contradicts the fact that the vertices are sorted according 
to the Prim order: indeed, by definition, the extremity of the lightest edge out of 
$\cup_{i\leq x} I_i$ is $a_{x+1}\in I_{x+1}$.
\end{proof}

Denote by $\Neigh^G(v)=\{u~: \{u,v\}\in E\}$ the set of neighbours of $v$ in $G$. 
For a set of nodes $S$ let also $\Neigh^G(S)=\l(\bigcup_{v \in S} \Neigh^G(v) \r)\setminus S$, 
the set of neighbours of $S$ (out of $S$). 
Aldous \cite{aldous1997} study  of the multiplicative coalescent relies on an exploration 
of the graph and an encoding of the process $i\mapsto \#\Neigh^G(S_i)$, for an increasing collection
of sets $(S_i)_{1\le i\le n}$ that are built by a breadth-first search algorithm. 
The modified Algorithm~\ref{alg:aldous} uses the standard order on the nodes instead
of breadth-first search, but here we investigate the influence of the order on $[n]$, hereafter 
denoted by $<$, that is used in building the sets $(S_i)_{1\le i\le n}$. The exploration is as follows.

The first visited node is the smallest one $v_1$ for the order $<$. Assume we have visited 
$S_k=\{v_1,\dots, v_{k}\}$ at some time $1\leq k\leq n$. 
Then two cases arise:
\begin{itemize}
\item if $\Neigh^G(S_k)\neq \varnothing$, then $v_{k+1}$ is the smallest node for 
$<$ in $\Neigh^G(S_k)$, or
\item if $\Neigh^G(S_k)= \varnothing$, then $v_{k+1}$ is the smallest node for 
$<$ in $[n]\setminus S_k$.
\end{itemize}
In the exploration used by Aldous, the labels of the nodes $1,2,\dots,n$ are compared using the 
standard order $<$ on $\mathbb{N}$. 
The exploration clearly depends on the order $<$, and the notation should have reflected
this fact. For example, we could have written $v_{k}(<)$ and  $S_k(<)$ instead of $v_k$ and $S_k$. 
We will sometimes used further these enriched notation, and for $0\leq k \leq n+1$ we also use the 
more compact notation
\[Z^{g}_{<}(k)=\#\Neigh^{g}(S_k(<))\]
where by convention $Z^{g}_{<}(0)=Z^{g}_{<}(n+1)=0$.

The proof of the following lemma is immediate. For a graph $g=([n],E)$, denote by $\CC(g)$ 
the set of connected components of $g$, seen as a partition of $[n]$. 
\begin{lem}\label{lem:rel} Let $g=([n],E)$ be a graph (connected or not). For any total order 
$<$ on the set of nodes $[n]$,
\[\#\{k \in \{0,\dots,n\}~:Z^g_{<}(k)=0\}=\#\CC(g),\]
and the successive sizes of the connected components ordered by the exploration coincide 
with the distances between successive zeros in the sequence $(Z^g_{<}(k),0\leq k \leq n+1)$.
\end{lem}

In the following we will call \emph{Prim exploration} the exploration based on the Prim order $\prec$. 
Unlike the standard exploration, it is defined only on properly weighted graphs $(G,\bw)$.  

\subsection{Prim traversal versus standard traversal}

Take a random graph ${\bf G}=([n],E)$ whose edges are equipped with i.i.d.\ uniform [0,1] weights. 
Let us examine the similarities and the differences between the standard exploration 
(using the order $<$) and the Prim exploration of the random graph ${\bf G}_t$.
By Lemma \ref{lem:rel} the multiset of excursions lengths of $Z^{{\bf G}_t}_<$ and  $Z^{{\bf G}_t}_{\prec}$ are 
the same, but in general the paths $Z^{{\bf G}_t}_<$ and  $Z^{{\bf G}_t}_{\prec}$ 
do not have the same distribution. We already said that the distributions of the corresponding 
processes in $t$ were different (by Proposition \ref{pro:Prim-perc}), 
but this is also true for fixed $t$ (even if the distribution of ${\bf G}$ is invariant by 
random permutation of the node labels). 
The reason is that during the exploration, the Prim order favours  the nodes with a 
large indegree since the order is defined using the weights of the edges.
 Here is an example illustrating this. Consider  ${G}$ the graph with vertices $\{1,a_1,a_2,a_3\}$ and edges 
$(1,a_1), (1,a_2), (1,a_3), (a_2,a_3)$. Conditionally on ${\bf G}_t=G$, one sees that for the standard exploration, 
under a random labelling preserving $1$, one visits $1,a_2,a_3,a_1$ in that order with probability 
$1/6$. However, under the Prim exploration, it is easy to check that, among the 24 possible 
orderings of the weights $(w_e,e\in E)$,  six of them give this order so that the probability is 
$1/4$. 

There are however some special cases, including when ${\bf G}$ is a uniform Cayley tree or the 
complete graph for which the distributions of $Z^{{\bf G}_t}_<$ and  $Z^{{\bf G}_t}_{\prec}$ 
are the same. 

\begin{lem}\label{lem:red} Let $({\bf G},W)=(([n],E),W)$ be a rooted weighted random graph. 
Assume that, for any $k$ the distribution of $\#\Neigh^{{\bf G}_t}(S_{k+1}(\prec))$ knowing 
$(S_k,\Neigh^{{\bf G}_t}(S_k(\prec)))$ is
\begin{itemize}
\item independent of the weights $w_e$ on the edges between $S_k$ and $\Neigh^{{\bf G}_t}(S_k(\prec))$, 
\item the same as the distribution of $\#\Neigh^{{\bf G}_t}(S_{k+1}(\prec))$ 
knowing $(S_k,\#\Neigh^{{\bf G}_t}(S_k(\prec)))$, 
\end{itemize}
then  $Z^{{\bf G}_t}_\prec$ and  $Z^{{\bf G}_t}_{<}$ have the same distribution. 
\end{lem}
\begin{proof}
We prove by induction on $k\ge 0$ that under the conditions of the lemma, for any $t\in [0,1]$,
$(Z_\prec^{\bG_t}(i))_{0\le i\le k}$ and $(Z_<^{\bG_t}(i))_{0\le i\le k}$ have the same distribution. 
The base case $k=0$ is clear. Suppose now that this holds up to some integer $k$. Then, by
Skorokhod's representation theorem, we can find a coupling for which $(Z_\prec^{\bG_t}(i))_{0\le i\le k}$ 
and $(Z_<^{\bG_t}(i))_{0\le i\le k}$ are a.s.\ the same. Now, the distribution of $\#\Neigh(S_{k+1})$ 
conditional on $(S_k,\#\Neigh(S_k))$ is the same as the one conditionally on $(S_k, \Neigh(S_k))$, for both orders. 
Furthermore, since this distribution is independent of the weights between $S_k$ and $\Neigh(S_k)$, it is not 
affected when modifying them in such a way that the end point of the lightest edge is the node of 
minimum label in $\Neigh(S_k)$. But in this modified version, we then have $Z_\prec^{\bG_t}(k+1)=Z_<^{\bG_t}(k+1)$
with probability one, so that $(Z_\prec^{\bG_t}(i))_{0\le i\le k+1}$ and $(Z_<^{\bG_t}(i))_{0\le i\le k+1}$,
which completes the proof of the induction step.
\end{proof}

\begin{lem}\label{lem:prpr} The following models both satisfy the  hypotheses (and then the conclusion) 
of Lemma~\ref{lem:red}:\\
$(a)$ The complete graph ${\bf G}=K_n$ on $[n]$ whose edges are weighted with i.i.d.\ uniforms on $[0,1]$.\\
$(b)$ ${\bf G}$ a uniform Cayley tree on $n$ nodes whose edges are equipped with i.i.d.\ weights
uniform on $[0,1]$.  
\end{lem}
Case $(b)$ can easily be extended to Galton--Watson trees conditioned to have size $n$, but 
we omit the details (see the proof below). 

\begin{proof} In the entire proof, there is no risk of confusion and we write $S_k$ instead of $S_k(\prec)$, 
and drop the superscript referring to the graph we are working on.

(a) The case of the complete graph is straightforward: whatever the node $v_{k+1}$, 
the distribution of $\#\Neigh(S_{k+1})$ given $(S_k, \Neigh(S_k)$ is always the same and is that of 
\[\#\Neigh(S_k-\I{\#\Neigh(S_k)>0} + {\sf Bin}\big(n-k-\I{\#\Neigh(S_k)>0}-\#\Neigh(S_k), t\big).\]
In particular, it is independent of the weights on the edges between $S_k$ and $\Neigh(S_k)$. 
Also, this distribution is the same conditionally on $(S_k, \Neigh(S_k))$.

(b) The case of the Cayley tree is based on the invariances of the distribution of the tree. 
Observe first that the percolated tree $\bG_t$ is distributed as a forest of Cayley trees, which
may each be seen as rooted at the node of smallest label.
Now, condition on $(S_k, \#\Neigh(S_k))$. If $\#\Neigh(S_k)=0$, then the claim clearly holds. 
Otherwise, the distribution of $\#\Neigh(S_{k+1})$ is that of 
\[\#\Neigh(S_k)-1 + {\sf Bin}(D_{k+1}, t),\]
where $D_{k+1}+1$ denotes the degree of $v_{k+1}$, the next node to be visited.
However, the subtrees rooted at $\Neigh(S_k)$ are exchangeable, and since the weights are 
independent of the tree, we see that the distribution of $D_{k+1}$ is independent of the 
weights between $S_k$ and $\Neigh(S_k)$. 
Furthermore, the distribution conditionally on $S_k$ and $\Neigh(S_k)$ is unchanged since 
we may still permute the sets of children of the nodes in $\Neigh(S_k)$ without altering 
the conditional distribution. 
\end{proof} 

\section{Encoding the additive coalescent: Proof of Theorem~\ref{theo:main+}}\label{sec:CVFDDyplus}

\subsection{Finite-dimensional distributions}

We start with the proof of the convergence of the finite-dimensional distributions (fdd).
\begin{lem}\label{lem:fdd+}
For any integers $k\ge 1$ and $\ell\ge 1$, and for any $s_1,s_2, \dots, s_k\in [0,1]$ and 
$\lambda_1,\lambda_2, \dots, \lambda_\ell \geq 0$, we have 
\[\big(\sy^{n,(\lambda_j)}_+(s_i)\big)_{1\le i\le k, 1\le j\le \ell} 
\dd 
\big(\sy^{(\lambda_j)}_+(s_i)\big)_{1\le i\le k, 1\le j\le \ell}.
\]
\end{lem}
We start with two bounds that will be used all along the proof. First, for any $`e>0$,
\ben\label{eq:conv1}
`P\l(\max_{1\leq i \leq n} X^{n,(1)}_+(i)\geq n^{`e}\r)=O\l( \exp\l(-n^{`e/2}\r)\r),
\een
as $n\to\infty$. To see this, observe that $X^{n,(1)}_+(i)$ are Poisson$(1)$ random variables 
conditioned to satisfy $\tau_{-1}=n$, and we have $`P(\tau_{-1}=n) \sim c n^{-3/2}$ as $n\to\infty$.  
So the conditioning may be removed at the expense of a factor 
$(cn^{-3/2})^{-1}$, the union bound brings another factor $n$, and then 
$`P({\sf Poisson}(1)\geq n^{`e})\leq \min\{ e^{-1+e^{s}-s n^{`e}}~:s>0\}$ by standard 
Chernoff's bounding method. The case $\lambda=1$ provides the bound in \eqref{eq:conv1}. 
Furthermore, as a consequence of the weak convergence in $\bbC([0,1],\bbR)$ stated in \eref{eq:yxn},
we have 
\ben\label{eq:conv2}
\big\|\sy^{n,(0)}_+\big\|_\infty \to \big\|\sy^{(0)}_+\big\|_\infty,
\een
and $\|\sy^{(0)}_+\|_\infty<\infty$ with probability one. 

\begin{proof}[Proof of Lemma~\ref{lem:fdd+}]
By the Skorokhod representation theorem, there exists a probability space on which the convergence stated in 
\eref{eq:yxn} is almost sure. In the following, we work on this space, and keep the same notation 
for the version of $\sy^{n,(0)}_+$ which converges a.s., and still denote the discrete increments by 
$(X_+^n(i),1\leq i \leq n)$.  Conditionally on $(X_+^n(i),1\leq i \leq n)$, we have
\begin{equation}\label{eq:reps}
\sy^{n,(\lambda)}_+(s)=n^{-1/2}\sum_{m=1}^{\floor{ns}} 
\Bigg(-1+\sum_{\ell=1}^{X^n_+(m)} \I{U_m(\ell)\leq 1-\lambda n^{-1/2}}\Bigg)+`e_{n,s,\lambda},
\end{equation}
where $`e_{n,s,\lambda}$ is the term coming from the fact that $ns$ is not an integer, so that the sum 
misses a contribution corresponding to the portion between $\floor{ns}$ and $ns$. 
We have the simple bound
\ben\label{eq:dns}
\sup_{0\le s\le 1,\lambda \in \bbR}|`e_{n,s,\lambda}|\leq n^{-1/2}\max_{1\le i\le n} |X^n_+(i)-1|.
\een
Note that, by \eref{eq:conv1}, the bound in \eqref{eq:dns} goes to 0 in probability. 
Using the fact that $\sum_{m=1}^{\floor{ns}}X^n_+(m)=\floor{ns}+\sqrt{n}\sy^{n,(0)}_+(s)$, 
we may rewrite $\sy^{n,(\lambda)}_+(s)$ as:
\begin{align}\label{eq:qfsq}
\sy^{n,(\lambda)}_+(s)
&=n^{-1/2}\sum_{m=1}^{\floor{ns}} \sum_{\ell=1}^{X_+^n(m)} 
\l(\I{U_m(\ell)\leq 1-\lambda n^{-1/2}}-(1-\lambda n^{-1/2})\r) \\
\nonumber   
& \quad -n^{-1/2}\floor{ns}+(1-{\lambda}n^{-1/2}) n^{-1/2}\l(\floor{ns}+\sqrt{n}\sy^{n,(0)}_+(s)\r)\\
\nonumber   & \quad +`e_{n,s,\lambda}.
\end{align}
Now, for $\lambda\in \bbR$ fixed, the first term goes to 0 in probability (conditionally on the $X_i$). 
To see this, observe that the number of terms in the sum is $\floor{ns}+\sqrt{n}\sy^{n,(0)}_+(s)$, 
and that the terms $\I{U_i(\ell)\leq 1-\lambda n^{-1/2}}-(1-\lambda n^{-1/2})$ are independent and each have 
variance bounded by $\lambda n^{-1/2}$. So the total variance of this first term is $O(n^{-1/2})$. 
Since $\| \sy^{n,(0)}_+\|_\infty$ converges a.s.\ on the probability space we are working on, 
and that each term is centred, Chebyshev's inequality guarantees that we have convergence to 
zero in probability. 

The term in the second line of the right-hand side of \eqref{eq:qfsq} is  
\ben\label{eq:qdsq}
(1-{\lambda}n^{-1/2})  \sy^{n,(0)}_+(s) 
-\lambda \frac{\floor{ns}}{n}  & \dd & \sy^{(0)}_{+}(s) -s\lambda . 
\een
Finally, we see that on the probability space we are working on for any $s,\lambda$,
\[ \sy ^{n,(\lambda)}_+(s) \proba \sy_+^{(0)}(s)-\lambda s.\] 
This completes the proof of convergence of the fdd. 
\end{proof}

\subsection{Tightness of the sequence $(\sy_+^n)_{n\ge 1}$} 
\label{sec:tightnesssyplus}

It suffices to show that the sequence $(\lambda\mapsto \sy^{n, (\lambda)}_+)_{n\ge 1}$ 
is tight on $\bbD([0,a], \bbC([0,1],`R))$ for each $a\geq 0$. So fix $a>0$. 
According to Kallenberg \cite[Theorem 14.10]{Kal}, considering the modulus of continuity
\[\omega_{\delta}(\sy^{n}_+):=
\sup\l\{ \|\sy^{n,(\lambda_1)}_+-\sy_+^{n,(\lambda_2)} \|_\infty~: |\lambda_1-\lambda_2|\leq \delta, 
0\leq \lambda_1,\lambda_2 \leq a\r\},\]
it suffices to show that for any $`e,`e'>0$, there exists $\delta>0$ such that for all $n$ large enough,
$`P(\omega_{\delta}(\sy^{n}_+)\geq `e) \leq `e'$.
Observe that this modulus of continuity is a priori not adapted to convergence in a space of cadlag
functions, but we take advantage of the continuity of the limit, 
and simply show that the convergence is uniform on $[0,a]\times[0,1]$. 
So fix $`e,`e'>0$. Considering again \eref{eq:conv1} and \eref{eq:conv2}, 
there exist constants $b\in (0,1/4)$ and $b'>0$ such that the following event
\[B_n:=\l\{\max_{1\leq i \leq n} X^{n,(1)}_+(i)\leq n^{b}, \big\|\sy^{n,(0)}_+\big\|_\infty \leq b'\r\}\]
has probability $1-`e'/10$ for all $n$ large enough. 

From \eref{eq:qfsq} and the discussion just below it, taking $\lambda_1 < \lambda_2$, we have
\begin{align}\label{eq:qfsq2}
\sy^{n,(\lambda_1)}_+(s)-\sy_+^{n,(\lambda_2)}(s)
& =n^{-1/2}\sum_{m=1}^{\floor{ns}} \sum_{\ell=1}^{X_+^n(m)} 
\l(\I{\lambda_1 n^{-1/2} <1-U_m(\ell)\leq \lambda_2 n^{-1/2}}-\frac{\lambda_2-\lambda_1}{n^{1/2}}\r) \\
\nonumber   &\qquad + s(-\lambda_1+\lambda_2)  
+ O\bigg( \frac{\|y^{n,(0)}\|_\infty}{n^{1/2}}+\frac{1}{n} +\frac{\max_i |X^n_+(i)-1|}{n^{1/2}}\bigg),
\end{align}
where the $O(\,\cdot\,)$ term above is uniform on $0\leq \lambda_1,\lambda_2\leq a, 0\leq s\leq 1$.  
We thus have $
w_{\delta}(\sy^{n})\leq w_{\delta}({\sf v}^{(n)})+ O(s(\lambda_2-\lambda_1)),$
where 
\[{\sf v}^{n,(\lambda)}(s)=n^{-1/2}\sum_{m=1}^{\floor{ns}} \sum_{\ell=1}^{X_+^n(m)} 
\l(\I{U_m(\ell)\leq 1-\lambda n^{-1/2}}-(1-\lambda n^{-1/2})\r) .\]
On $B_n$, the $O(\,\cdot\,)$ term in \eqref{eq:qfsq2} goes to zero in probability. 
Furthermore, since $`P(\complement B)<`e'/10$, in order to complete the proof, it suffices to show
the following lemma:
\begin{lem}For every $`e,`e'>0$ there exists $\delta>0$ such that for every $n$ large enough
\[`P(\omega_{\delta}({\sf v}^{n})\geq `e/2) \leq `e'/2.\]
\end{lem}
In fact, ${\sf v}^{n,(\lambda)}(s)$ appears to be very small since it is a sum of 
$ns + \sqrt{n}\sy_+^{n,(0)}(s)$ centred r.v.\ with variance $1/\sqrt{n}$ and the normalisation 
by $n^{-1/2}$ makes of the sum a centred r.v.\ with variance $1/\sqrt{n}$; however the 
fluctuations along the two dimensional rectangle are more difficult to handle. 
\begin{proof}
We discretise the parameter $\lambda$ and consider $\lambda_j= j a /\sqrt{n}$ for $j=1,\dots,n.$ Then, 
\[{\sf v}^{n,(\lambda_{j+1})}(s)-{\sf v}^{n,(\lambda_j)}(s)
=n^{-1/2}\sum_{m=1}^{\floor{ns}} 
\sum_{\ell=1}^{X_+^n(m)}\l(\I{ 1-\frac{ (j+1)a}{ n}\leq U_m(\ell)\leq 1-\frac{ja}n} -\frac{a}{n}\r).\]
Now,  ${\sf v}^{n,(\lambda)}$ does not fluctuate much between the any two successive 
$\lambda_j$, $j=1,\dots,n$:
\begin{align*}
\sup_{1\leq j\leq n\atop{\lambda \in [\lambda_j,\lambda_{j+1}]}}|\sv^{n,(\lambda)}(s)-\sv^{n,(\lambda_j)}(s)|
\leq& \l(\floor{ns}+\sqrt{n}\sy_+^{n,(0)}(s)\r)\frac{a}{n^{3/2}}\\
  &+\sup_{1\leq j\leq n}\frac{\#\l\{1-U_m(\ell)\in\big[\frac{ ja}{ n},\frac{ (j+1)a}{ n}\big], 
  m\in\cro{1,n}, \ell\in\cro{1,X_+^{n}(m)}\r\}}{n^{1/2}}.
\end{align*}
The first term goes to 0. As for the second one, for a single $j$, on $B_n$ the term is dominated by
$n^{-1/2}{\sf Bin}(n+n^{1+b},a/n)$, where ${\sf Bin}(n,p)$ denotes a binomial r.v.\ with parameters $n$ and $p$. 
Now, using Bernstein's inequality, one can show, that for some $c>0$, and $n$ large enough,
\[`P\l(\frac{{\sf Bin}(n+n^{1+b},a/n)}{n^{1/2}}>`e/4\r)\leq  \exp(-c`e^2n^{1/2}),\]
so that the union bound then suffices to control the supremum on $j\in\{1,2,\dots,n\}$. 
It remains to bound the fluctuations restricted to the $\lambda_j$, $1\le j\le n$. 

Now, let $(\lambda,s)\mapsto \sw^{n,(\lambda)}(s)$ be the continuous process 
(in $\bbC([0,a]\times[0,1],`R)$) which interpolates ${\sf v}^{n}$ as follows: 
for any $1\leq j \leq n$, $\sw^{n,(\lambda_j)}= {\sf v}^{n,(\lambda_j)}$. 
Each $\sw^{n,(\lambda_j)}$ is interpolated between the points ($k/n,(k+1)/n)$ 
(this corresponds to the discrete increments), and the interpolation from 
$\sv^{n,(\lambda_j)}(s)$ to $\sv^{n,(\lambda_{j+1})}(s)$ being then also linear for each $s$.
To end the proof of tightness, we will show that  $\sw^{n}$ is tight in $\bbC([0,a]\times[0,1],`R)$.  
We use the criterion given in Corollary~14.9 p.~261 of \cite{Kal}. It suffices to show that  
$(\sw^{n,(0)}(x))_{x\in[0,1]}$ is tight (which is the case since $\sy_+^{n,(0)}$ is tight), 
and for some constant $C$, some $\alpha,\beta>0$, for any $(\lambda_1,s_1)$ and  $(\lambda_2,s_2)$ 
in $[0,a]\times [0,1]$, any $n$ large enough,
\ben\label{eq:tty}
`E\Big[ \l|\sw^{n,(\lambda_1)}({s_1})-\sw^{n,(\lambda_2)}(s_2)\r|^\alpha \Big] 
\leq  C\|(s_1,\lambda_1)-(s_2,\lambda_2)\|^{2+\beta}.
\een
We will show this conditionally on $B_n$ only, which is sufficient too. 
As usual, it suffices to get the inequality at the discretisation points. 
\[`E \bigg[\l|\sw^{n,(\lambda_{j_1})}\bigg(\frac{k_1}n\bigg)-
\sy_+^{n,(\lambda_{j_2})}\bigg(\frac{k_2}n\bigg)\r|^p\bigg]
=
`E\bigg[\bigg|
n^{-1/2}\sum_{m=k_1+1}^{k_2} \sum_{\ell=1}^{X_+^n(m)} 
\l(\I{j_1 a/n <1-U_m(\ell)\leq j_2 a/ n}-a\frac{j_2-j_1}{n}\r)\bigg|^p\bigg]\]
Using that for any sequence of independent and centred random variables $(A_i)_{i\ge 1}$, we have 
$`E[|\sum_{k=1}^m A_k|^p]\leq C_p`E[\sum_{k=1}^m A_k^2]^{p/2}$, for some constant $C_p$, we obtain
\begin{align*}
`E\bigg[\l|\sw^{n,(\lambda_{j_1})}\bigg(\frac{k_1}n\bigg)-
\sy_+^{n,(\lambda_{j_2})}\bigg(\frac{k_2}n\bigg)\r|^p~\bigg|~\sy_+^{n,(0)}\bigg]
&\leq {C_p} \l\{\frac{k_2-k_1}n+n^{-1/2}\l(\sy_+^{n,(0)}\bigg(\frac{k_2}n\bigg)
-\sy_+^{n,(0)}\bigg(\frac{k_1}n\bigg)\r)\r\}^{p/2}\\
&\qquad \times \bigg(\frac{(j_2-j_1)a}{\sqrt{n}}\bigg)^{p/2}\frac{1}{\sqrt{n}^{p/2}}
\end{align*}
this is much smaller than needed: conditionally on $B_n$ this is at most
\[\frac{C_1}{n^{p/2}} \l(\lambda_{j_1}-\lambda_{j_2}\r)^{p/2}\l(\frac{k_2-k_1}{n}\r)^{p/2}
+ {C_2}  \l(\lambda_{j_1}-\lambda_{j_2}\r)^{p/2}  
\l(\frac{\sy_+^{n,(0)}(k_2/n)-\sy_+^{n,(0)}(k_1/n)}n\r)^{p/2}.\]
Finally, since $\sy_+^{n,(0)} \leq b'$ on $B_n$, we have 
\[\l(\frac{\sy_+^{n,(0)}(k_2/n)-\sy_+^{n,(0)}(k_1/n)}n\r)^{p/2}
\leq \l(\frac{M}{n}\r)^{p/2}
\leq  C'\l(\frac{k_2-k_1}{n}\r)^{p/2},\] 
which completes the proof of lemma.
\end{proof}

\section{Encoding the multiplicative coalescent: Proof of Theorem~\ref{theo:mainx}}
\label{sec:enc_mult}

Theorem \ref{theo:mainx} states the convergence of $\sy^{n}_\times$. 
Before proving this convergence, we need to establish carefully the distribution of the 
sequence $\sy^n_\times$. This is the aim of the next subsection. We then move on 
to the proof of convergence in the sense of finite-dimensional distributions in Section~\ref{sec:fdd_yx}
and of tightness of the sequence $(\sy^n_\times)_{n\ge 1}$ in Section~\ref{sec:tightness_yx}.

\subsection{Distribution of $\sy^{n }_\times$}
\label{sec:dissyn}

In this part, we work on ${\bf G}=K_n$ equipped with the weights ${\bf w}=(W(i,j), ij\in [n])$ on its edges, 
some i.i.d. uniform r.v. on $[0,1]$. Let $u_1,\dots,u_n$ be the list of nodes of $[n]$ 
sorted according to Prim's order on $({\bf G},{\bf w})$. We now let $U_{k}(\ell)= W(u_k,u_\ell)$.
In this case ${\bf G}_t$ coincides with $G(n,t)$, and Lemmas~\ref{lem:prpr} and~\ref{lem:red} apply.

For $i\ge 1$ and $t\in [0,1]$, define 
\begin{align*}
\cZ^t(i) &:= \big\{ u_k : i< k\le n, \exists j\le i, U_j(k) \le t \big\}\\
\cS^t(i) & :=\big\{ u_k : i< k\le n,  U_i(k)\le t\} \setminus \cZ^t(i),
\end{align*}
and let $Z^t(i)=\#\cZ^t(i)$ and $S^t(i)=\#\cS^t(i)$. The set  $\cZ^t_i$ is 
the list of nodes with Prim index greater than $i$, that share an edge in ${\bf G}_t$ 
with a node with index at most $i$. The set $\cS^t(i)$ records the neighbours of $u_i$ in ${\bf G}_t$ with Prim index larger than $i$, that are not neighbours of any $u_j$ for $ j < i$. 
Observe that for every $i$ the map $t\mapsto\cZ^t(i)$ is non-decreasing, non-negative, 
and that $\cZ^t(0)=\varnothing$ for every $t\in [0,1]$. Then define
$\Delta Z^t(i):=Z^t(i)-Z^t(i-1)$; in ${\bf G}_t$, this is the number 
of nodes in $\{u_{i+1},...,u_n\}$ that are adjacent to $u_i$ but are not 
to any of $u_1,\dots,u_{i-1}$, minus the number of nodes in 
$\{u_{i},\dots,u_{n}\}$ with an edge from $u_1,\dots,u_{i-1}$ but not from $u_i$. 
Hence $\Delta Z^t(i)\geq -1$ since only $u_i$ can be in $\cZ^t(i)$ but not in $\cZ^t(i-1)$. 

For a vector $\bx=(x_1,x_2,\dots, x_k)$ of distinct real numbers,
we write $\bx^{\downarrow}$ for the vector consisting of the elements $x_1,\dots, x_n$ 
sorted in decreasing order.

\begin{lem}\label{lem:dist_incr}
For every $i\in [n]$, conditionally on $(\cZ^t(i-1))_{t\in [0,1]}$, we have 
\begin{enumerate}[$(a)$]
\item $(U_i(j))^{\downarrow}_{i< j\le n}$ is the reordering in decreasing order of $n-i-1$ 
i.i.d.\ $[0,1]$-uniform random variables,
\item $(\Delta Z^t(i))_{t\in [0,1]} = 
(\#\big\{k : U_i(k)\le t, i<k\le n, k\not\in \cZ^t(i-1) \big\}
-\I{\cZ^t(i-1)\ne\varnothing})_{t\in [0,1]}$, and 
\item $(S^t(i))_{t\in [0,1]} 
= (\#\big\{k : U_i(k)\le t, i<k\le n, k\in \cZ^t(i-1) \big\})_{t\in [0,1]}.$ 
\end{enumerate}
\end{lem}
\begin{proof}
We define the sets $(\cZ^t(i),1\leq i \leq n)$ and the Prim ordering 
$(u_1,u_2,\dots, u_n)$ simultaneously. For $i\ge 1$, $V_i$ denotes the set of nodes 
$\{u_1,\dots, u_i\}$. 
We proceed by induction on $i\ge 1$ and prove (a), (b) and (c) simultaneously for every 
$t\in [0,1]$. Initially, we have $\cZ^t(0)=\varnothing$ and $i=1$, $u_1=1$, $V_1=\{u_1\}$. 
The weights $(W(u_1,k))_{1<k\le n}$ of the $n-1$ edges which have $u_1$ as an end point are 
independent and uniformly distributed on $[0,1]$, and by definition of ${\bf G}_t$, for every $t\in [0,1]$,
\begin{align*}
\Delta Z^{t}(1) 
& = \#\big\{v \in [n]\setminus V_1: W(1,v)\le t\big\}\\
& = \#\big\{k \in [n]\setminus V_1: U_1(k)\le t, k\not\in \cZ^{t}(0)\big\},
\end{align*}
which proves the base case since $\cZ^t(0)=\varnothing$, for every $t\in[0,1]$. Furthermore, 
$\cS^t(1)\subseteq \cZ^t(0) = \varnothing$. 
This proves all three claims for every $t\in[0,1]$, in the case that $i=1$. 

Suppose now that the claims (a), (b) and (c) hold true for all $j\in \{1,2,\dots, i\}$, 
and consider Prim's algorithm just after $u_{i}$ has been defined. 
By definition, the node $u_{i+1}$ is the node $v$ in $[n]\setminus V_i$ for which the 
distance $d(v, V_{i})$ is minimised. In particular, conditional on $(\cZ^t(i))_{t\in [0,1]}$,
\[d(u_{i+1},V_i)=\inf\{t\ge 0: \cZ^t(i)\ne \varnothing\}.\]
Observe that the choice of $u_{i+1}$ is done by looking at the weights of the edges $W(u_j,v)$ for $j\le i$ 
and $v\not\in V_i$. In particular, this choice is independent of the weights 
$(W(u_{i+1},k): k\not \in V_{i+1})$, so that these weights are also i.i.d.\ uniform 
on $[0,1]$, conditionally on the nodes rank $u_1,\dots,u_{i+1}$. 
This is even true conditionally on $(\cZ^t(i))_{t\in [0,1]}$ since these
random variables only depend on the weights $W(u_j,v)$, $1\le j\le i$, $v\in [n]$. 
Observe that once $u_{i+1}$ is defined, the collection (the set) of weights to its 
neighbours in $[n]\setminus V_{i+1}$ is fixed, even if we do not know yet the 
precise Prim order induced on $[n]\setminus V_{i+1}$, so (a) follows readily. 

To prove (b), observe that we have the following disjoint union (denoted by $\bigsqcup$),
\begin{align}\label{eq:disjoint_union}
\cZ^t(i+1) 
& = \big( \cZ^t(i) \setminus V_{i+1} \big) \bigsqcup \,\, \big( 
\big\{v\in [n]\setminus V_{i+1} : W(u_{i+1}, v)\le t\big\} \setminus \cZ^t(i) \big)\notag\\
& = \big( \cZ^t(i) \setminus V_{i+1} \big)
\bigsqcup \,\, \big\{k: U_{i+1}(k)\le t, i+1<k\le n, k\not\in \cZ^t(i)\big\},
\end{align}
which expresses the fact that we canonically assign the elements $v$ of $\cZ^t(i)$ to the first 
node $u_j\in V_i$ for which $W(u_j,v)\le t$. 
The first set of \eqref{eq:disjoint_union} is easy to deal with. Indeed, by 
Proposition~\ref{pro:Prim-perc}, we have: 
\begin{itemize}
     \item if $\cZ^t(i)=\varnothing$ none of the nodes in $V_i$ is connected to 
     any of the nodes in $[n]\setminus V_i$ by an edge of $E_t$; 
     \item if $\cZ^t(i)\ne \varnothing$ some nodes of $V_i$ are connected to some nodes in $[n]\setminus V_i$. 
     But then, by Proposition~\ref{pro:Prim-perc}, $u_{i+1}\in \cZ^t(i)$ 
     and there are $Z^t(i)-1\ge 0$ nodes of $[n]\setminus V_{i+1}$ which are connected 
     to some node of~$V_i$. 
\end{itemize} 
So, in any case, there are $(Z^t(i)-1)_+$ nodes of the set $[n]\setminus V_{i+1}$ 
which are already connected to nodes in $V_i$ by edges of $E_t$, and we have
\begin{equation}\label{eq:card_forbidden}
|\cZ^t(i) \setminus V_{i+1}| = (Z^t(i)-1)_+.
\end{equation}
The representation of (b) follows immediately, and (c) is straightforward from the definition.
\end{proof}

\begin{cor}\label{cor:dist_incr}There exists a collection $(U_i^\star(j))_{1\le i<j\le n}$
of i.i.d.\ random variables uniformly distributed on $[0,1]$ such that, for every 
$i\in [n]$, the family $(U^\star_i(j))_{i\le j\le n}$ is independent of $(Z^t(i))_{t\in [0,1]}$. 
Furthermore, conditionally on $(Z^t(i))_{t\in [0,1]}$, we have for every $t\in [0,1]$
\begin{enumerate}[$(a)$]
\item $\Delta Z^t(i) = 
\#\big\{k : U^\star_i(k)\le t, i+(Z^t(i-1)-1)_+<k\le n \big\} -\I{Z^t(i-1)>0}$, and 
\item $S^t(i) 
= \#\big\{k : U^\star_i(k)\le t, i<k\le i+(Z^t(i-1)-1)_+\big\}.$ 
\end{enumerate}
\end{cor}
\begin{proof}
The proof consists simply in observing that $\#(\cZ^t(i-1)\setminus V_{i})=(Z^t(i-1)-1)_+$ 
by \eqref{eq:card_forbidden}, and in reordering the random variables, $(U_i(j))_{i<j\le n}$ 
for every fixed $i$ in such a way that they are hit by the set $\cZ^t(i-1)$ in 
increasing order of index as time increases. So fix $i\ge 2$ and define 
$t^i_{i+1}=\inf\{t\ge 0: \#\cZ^t(i-1)\ge 1\}$. Note that a.s., $\cZ_{i-1}(t^i_{i+1})$ 
contains a single element which we denote by $\pi_i(i+1)$. Then, for $i<j<n$, 
let $t^i_{j+1}=\inf\{t> t^i_{j}: \#\cZ^t(i-1)>\#\cZ^{t-}(i-1)\}$ and define $\pi_i(j+1)$ 
to be the a.s.\ unique element of $\cZ^{t^i_{j+1}}(i-1)\setminus \cZ^{t^i_{j+1}-}(i-1)$. 
For $i<j\le n$, define $U^\star_i(j)=U_i(\pi_i(j))$. In order to conclude the proof, observe that 
the permutation $\pi_i$ is depends only on $(\cZ^t(i-1))_{t\in [0,1]}$, and is therefore independent 
of $(U_i(j))_{i<j\le n}$. 
\end{proof}
  
Using Lemma \ref{lem:rel},  the process $(Z^t(i),i\geq 0)$ seems to be a nice tool to study 
$\CC({\bf G}_t)$, but it is not since its convergence is not sufficient to entails the convergence of the sizes of the excursions . To circumvent this problem, the idea, already exploited by \cite{aldous1997} and \cite{CL}, 
is to use a companion process $Y$ which has a drift and for which the lengths of the 
excursions above the current minimum converge. 

For this, we use the settings of Corollary~\ref{cor:dist_incr}, and for $0\leq i \leq n, t\in[0,1]$ 
we set 
\begin{align}\label{eq:zedq}
X^{n,t}_\times(i)  
&= \#\big\{k : U^\star_i(k)\le t, i+(Z^t(i-1)-1)_+<k\le n \big\}  \\
Y_\times^{n,t}(i) &= \sum_{j=1}^i (X^{n,t}_\times(j)-1).
\end{align}

\begin{lem}For every $t\in [0,1]$, we have $Z^{n,t}_\times=\Psi  Y^{n,t}_\times$.
In particular,  by Lemma \ref{lem:Psi-exc} the excursions of $Y_\times^{n,t}$ above its current minimum coincide with the 
excursions of $Z^{n,t}$ away from zero. 
\end{lem}
\begin{proof} The processes $Z_\times^{n,t}$ and $Y^{n,t}_\times$ have the same increments, 
except when $Z^t_{i-1}=0$, in which case $\Delta Y^{n,t}_\times(i)=\Delta Z_\times^{n,t}(i)-1$. 
The conclusion follows easily.
\end{proof} 

We are now ready to prove the convergence of  $\sy^{n }_\times$. 
The proof is quite similar to that of Theorem \ref{theo:main+}.
Again it is sufficient to show the weak convergence in 
$\bbD([\lambda_\star,\lambda^\star],\bbC([0,a],`R))$ 
for $-\infty<\lambda_\star<\lambda^\star<+\infty$ and $a>0$ fixed. We start by revisiting the 
proof of Aldous \cite{aldous1997} of \eref{eq:mono-lambda}. Again, Lemma~\ref{lem:red} entails 
that $\sy^{n,(\lambda)}_\times$ has same distribution as 
Aldous' version $y^{n,(\lambda)}$ associated with breadth first order.
 
\subsection{Proof of convergence of $\sy^{n,(\lambda)}_\times$  for a fixed $\lambda$}
\label{sec:newproof_Aldous}

Here $\lambda$ is fixed, and we obtain the convergence in $\bbC([0,a],\bbR)$ of $\sy^{n,(\lambda)}_\times$.
Aldous  proved \eref{eq:mono-lambda} using the approximation of the Markov chain $Y^{n,(\lambda)}$ 
by a diffusion, the convergence being on $\bbD([0,1],\bbR)$; here, we provide a proof that is simpler 
and more easily extendable. For short, we use the following notation
\begin{align*}
X^{(\lambda)}(i)&=X^{n,p_\lambda(n)}_\times(i)\\
Y^{(\lambda)}(i)&=Y^{n,p_\lambda(n)}_\times(i)\\
Z^{(\lambda)}(i)&=Z^{n,p_\lambda(n)}_\times(i).
\end{align*}
Define also $\bar{X}^{(\lambda)}(i)$ and $\underline{X}^{(\lambda)}(i)$ by
\begin{align*}
\bar{X}^{(\lambda)}(i)
&=\#\big\{k : U^\star_i(k)\le p_\lambda(n), i <k\le n \big\}
\sim {\sf Bin}(n-i,p_\lambda(n))\\
\underline{X}^{(\lambda)}(i)
&=\#\big\{k : U^\star_i(k)\le p_\lambda(n), i+n^{1/2} <k\le n \big\}
\sim {\sf Bin}(n-i-n^{1/2},p_\lambda(n)).
\end{align*}
(The term $n^{1/2}$ in the definition of $\underline{X}^{(\lambda)}(i)$ may be replaced by $n^{\alpha}$,
for any $\alpha\in [1/3,2/3]$.)
Here, in Section~\ref{sec:newproof_Aldous}, $\lambda$ is fixed and we further shorten the notation by dropping 
the superscripts indicating its value. 
Then define $\bar{Y}(i)$ and $\bar{Z}(i)$ (resp.\ $\underline{Y}(i)$ and 
$\underline{Z}(i)$) with $\bar{X}$ (resp.\ $\underline{X}$) in the same way that 
$Y$ and $Z$ are defined with $X$ in \eqref{eq:zedq}.  
The random variables $X(i)$, $\bar X(i)$ and $\underline X(i)$ are all defined with the same uniforms, 
and, as long as $Z_i\leq n^{1/2}$, we have
\[\underline{X}(i)\leq X(i) \leq \bar{X}(i).\]
In particular, this inequality holds at least until the first time when $Z(i)$ exceeds $n^{1/2}$, 
which is no earlier than the first time $i$ when $\bar Z(i)\ge n^{1/2}$, that we denote by 
$\tau_{n^{1/2}}(\bar Z)$.
Now, consider some times $t_0:=0<t_1<\dots<t_\kappa\leq a$, for some $\kappa \ge 1$, and let us investigate 
the convergence of $(\Gamma(n^{2/3}t_j),1\leq j \leq \kappa)$ for $\Gamma=\bar{Y}$ and $W=\underline{Y}$. 
Set $\Delta t_j=t_j-t_{j-1}$ and $\Delta_2 t_j= t_j^2-t_{j-1}^2$. The increments 
$(\Delta \Gamma(n^{2/3}t_j):=\Gamma(n^{2/3}t_j)-\Gamma(n^{2/3}t_{j-1}),1\leq j \leq n)$ are independent, 
and for  $1\leq j \leq \kappa$, 
\begin{align*}
\Delta\bar{Y}(n^{2/3}t_j) 
&=\#\big\{k : U^\star_m(k)\le p_\lambda(n), m <k\le n, n^{2/3}t_{j-1}<m\leq n^{2/3}t_j \big\}
-n^{2/3}\Delta t_j \\
&\sim {\sf Bin}\bigg(n^{5/3}\Delta t_j-\frac{n^{4/3}}2\Delta_2 t_j+`e_1,p_\lambda(n)\bigg)
-n^{2/3}\Delta t_j
\end{align*}
and 
\begin{align*}
\Delta\underline{Y}(n^{2/3}t_j) 
&=\#\big\{k : U^\star_m(k)\le p_\lambda(n), m+n^{1/2} <k\le n, n^{2/3}t_{j-1}<m\leq n^{2/3}t_j \big\}
-n^{2/3}\Delta t_j\\
&\sim {\sf Bin}\bigg(n^{5/3}\Delta t_j-\frac{n^{4/3}}2\Delta_2 t_j
-n^{2/3+1/2}\Delta t_j+`e_2,p_\lambda(n)\bigg)-n^{2/3}\Delta t_j,
\end{align*}
where $`e_1$ and $`e_2$ account for the error made by replacing $k(k+1)^2/2$ by $k^2/2$ and by the 
approximation of the fractional parts; in particular, $`e_1,`e_2=O(n^{2/3})$ and eventually negligible.
Recall that $p_\lambda(n)=1/n+\lambda n^{-4/3}$. 
By the central limit theorem, for every $j\in \{1,\dots, \kappa\}$ we have
\[\frac{\Delta\bar{Y}(n^{2/3}t_j)
-\Delta\underline{Y}(n^{2/3}t_j) }{n^{1/3}}\dd 0\]
and 
\[\frac{\Delta\bar{Y}(n^{2/3}t_j)}{n^{1/3}}\dd \cN\bigg(\lambda(t_j-t_{j-1})
-\frac{(t_j^2-t_{j-1}^2)}{2},t_j-t_{j-1}\bigg)\]
where $\cN(\mu,\sigma^2)$ denotes the normal distribution with mean $\mu$ and variance $\sigma^2$.
This provides the convergence of the fdd of $\bar{Y}$ and $\underline{Y}$ 
to those of the process $(\sy^{(\lambda)}_{\times}(x), x\geq 0)$. This also implies that  
$\tau_{n^{1/2}}(\bar{Z})/n^{2/3}\to +\infty$. Then, the tightness of $Y$ follows from that 
of $\bar{Y}$ and $\underline{Y}$, and these ones are consequences of a simple control of the fourth moment of 
\[\frac{\Delta\bar{Y}(n^{2/3}t_j)}{n^{1/3}} 
-\bigg(\lambda(t_j-t_{j-1})-\frac{(t_j^2-t_{j-1}^2)}{2}\bigg)\]
which is $3(t_j-t_{j-1})^2+O(n^{-1/3})$ the $O(\,\cdot\,)$ term being independent of $0\leq t_j,t_{j-1}\leq a$ 
(we just used here that the fourth centred moment of ${\sf Bin}(n,q)$ is $nq(1-q)(1+(3n-6)(q-q^2))$. 
The same estimate with a different $O(\,\cdot\,)$ term holds for $\underline{Y}$. 
This completes the proof of \eref{eq:mono-lambda}. 

We now slightly adapt the proof to get the full convergence of the bi-dimensional process as stated in 
Theorem~\ref{theo:mainx}. We will prove the convergence in 
$\bbD([\lambda_\star,\lambda^\star]\times[0,a], \bbR)$ which will be sufficient to conclude. 

\subsection{Proof of the convergence of the fdd of  $\sy^{n}_\times$ to those of $\sy_\times$}
\label{sec:fdd_yx}

In the additive case, we used three main ingredients: the convergence of $\sy^{n,(0)}_+$ to 
$\sy^{(0)}_+$ (stated in \eref{eq:yxn}), a global bound on the increments of $\sy^{n,(\lambda)}_+$ 
(the bound in \eref{eq:conv1}), and a global bound on $(\sy^{n,(0)}_+,n\geq 0)$ (in \eref{eq:conv2}). 
The proof proceeding by comparing $\sy^{n,(\lambda)}_+$ with  $\sy^{n,(0)}_+$, and the expressing the limiting 
behaviour of $\sy^{n,(\lambda)}_+$ in terms of $\sy^{(0)}_+$, the limit of $\sy^{n,(0)}_+$.

Here we rely on the same ingredients. The first one is \eref{eq:mono-lambda}, 
taken at $\lambda^\star$
\ben\label{eq:mono-lambda2} 
\sy^{n,(\lambda^\star)}_\times \dd \sy^{(\lambda^\star)}_\times,
\een
and we will express the limit of $\sy^{n,(\lambda)}_\times$ for $\lambda \in[\lambda_\star,\lambda^\star]$ 
in terms of $\sy^{(\lambda^\star)}_\times$. The convergence \eref{eq:mono-lambda2} implies that
\[\sup_{x\in[0,a]}\l|\sy^{n,(\lambda^\star)}_\times(x)\r| 
\dd \sup_{x\in[0,a]}\l|\sy^{(\lambda^\star)}_\times(x)\r|<_{a.s.}+\infty.\]
We rely on the approximations $\bar{X}^{(\lambda)}$ and $\underline{X}^{(\lambda)}$ 
introduced in the previous section, and recall that all the variables $X,\bar{X},\underline{X}$ are all 
defined on the same probability space (the space where are defined the $U^\star$).

Clearly, we have
\[\sup_{0\leq i \leq an^{2/3}} X^{(\lambda)}(i)\leq \max_{0\leq i \leq an^{2/3}} \bar{X}^{(\lambda^\star)}(i),\] 
where the $\bar X^{(\lambda^\star)}(i)$, $1\le i\le an^{2/3}$, are independent 
${\sf Bin}(n-i,p_{\lambda^\star}(n))$ random variables. 
Thus, for $`e>0$ small,  
\begin{align*}
`P\bigg(\sup_{\lambda\in[\lambda_\star,\lambda^\star]}
\sup_{0\leq i \leq a n^{2/3}} X^{(\lambda)}(i) \geq n^{`e}\bigg)
&\le a n^{2/3} \cdot `P({\sf Bin}(n,p_{\lambda^\star}(n))\geq n^{`e})\\
&=O(\exp(-n^{\epsilon}/2)).
\end{align*}

Again by Skorokhod representation theorem, we consider a space where the convergence 
\eref{eq:mono-lambda2} holds with probability one.
Let us prove that  for $\lambda\in[\lambda_\star,\lambda^\star]$,
\ben\label{eq:joint}
\sup_{0\leq x \leq a} \l|\sy^{n,(\lambda^\star)}_\times(x)-\sy^{n,(\lambda )}_\times(x)
-\l(\lambda^\star-\lambda\r)x\r|\proba 0,\een
which suffices to get the convergence of the fdd. It suffices to get the same 
result with $\bar{\sy}^{n,(\lambda^\star)}_\times(x)-{\bar\sy}^{n,(\lambda )}_\times(x)$ 
and $\underline{\sy}^{n,(\lambda^\star)}_\times(x)-{\underline\sy}^{n,(\lambda )}_\times(x)$ 
instead (where $\bar \sy$ and $\underline \sy$ are defined from $\bar Y$ and $\underline Y$, respectively, 
in the same way that $\sy$ is defined from $Y$). But for these processes this follows from the fact that
\[\bar{Y}^{(\lambda^\star)}(n^{2/3}x)-\bar{Y}^{(\lambda)}(n^{2/3}x)
\sim {\sf Bin}\bigg(n^{5/3}x-\frac{n^{4/3}}{2}x^2+`e_n,p_{\lambda^\star}(n)-p_{\lambda}(n)\bigg).\] 
The same argument for ${\underline Y}^{(\lambda )}(n^{2/3}x)$ allows one to conclude.

\subsection{Tightness of the sequence $(\sy^{n}_\times)_{n\ge 1}$}
\label{sec:tightness_yx}

We mimic what is done for the proof of the tightness of $(\sy^{n}_+)_{n\ge 1}$ in Section~\ref{sec:CVFDDyplus}. 
We consider here the modulus of continuity 
\[w_{\delta}(\sy^{n}_\times)=\sup\big\{ \big\|\sy^{n,(\lambda_1)}_\times-\sy_\times^{n,(\lambda_2)} 
\big\|_\infty~: |\lambda_1-\lambda_2|\leq \delta, \lambda_\star\leq \lambda_1,\lambda_2 
\leq \lambda^\star\big\},\]
where $\|f\|_\infty=\sup\{|f(x)|:0\leq x\leq a\}$. Again, for any $`e>0$, 
any small fixed $b>0$,  the event 
\[B_n:=\l\{\max_{1\leq i \leq n} X^{n,(\lambda^\star)}_\times(i)\leq n^{b}, 
\big\|\sy^{n,(\lambda^\star)}_\times\big\|_\infty \leq n^{2/3}\r\}\]
has probability at least $1-`e$ for $n$ large enough.

As in the proof of the tightness of $\sy_+^{n,(\lambda)}$, one discretises the time 
parameter $p_{\lambda}(n)$. For $1\le j\le n$, we define 
\[ \lambda_j = \lambda_j(n) =  \lambda_{\star}+\frac{j}{n}(\lambda^\star-\lambda_\star).\]
When $j$ goes from 0 to $n$, $p_{\lambda[k]}(n)=1/n+\lambda_j/n^{4/3}$ goes from 
$p_{\lambda_\star}(n)$ to  $p_{\lambda^\star}(n)$. 
For $j\in\{0,...,n\}$, define $\sy^{n,[\lambda_j]}_\times=\sy^{n,(\lambda_j)}_\times$; then 
each one of the processes $\sy^{n,[\lambda_j]}$ is already continuous in $[0,a]$.
In order to obtain a process that is continuous for $\lambda\in[\lambda_j,\lambda_{j+1}]$ and 
$x\in [0,a]$, $\sy^{n,[\lambda]}_\times(x)$ is obtained  by linear interpolation 
of $\sy^{n,[\lambda_j]}_\times(x)$ and  $\sy^{n,[\lambda_{j+1}]}_\times(x)$.
Instead of proving the tightness of $\lambda\mapsto(x\mapsto \sy^{n,(\lambda)}_\times(x))$ 
in $\bbD([\lambda_\star,\lambda^\star],\bbC([0,a],\bbR))$, we prove the tightness of 
$(\lambda,x)\mapsto\sy^{n,[\lambda]}_\times(x)$ in $\bbC([\lambda_\star,\lambda^\star]\times[0,a],\bbR)$. 
For this we use a moment criterion, that is, a bound of the type 
\[
`E\big[\big| \sy^{n,[\lambda_2]}_\times(s_2)-\sy^{n,[\lambda_1]}_\times(s_1)\big|^{\alpha}\big]
\leq C \|(s_1,\lambda_1),(s_2,\lambda_2)\|^{2+\beta},\] 
for some norm $\|\,\cdot\,\|$, where $\alpha,\beta>0$ and $C$ is a constant.
Again, we get these bounds for $\bar{\sy}^{n}_\times$ and $\underline{\sy}^{n}_\times$ 
instead which is again sufficient to conclude.

However, for $\lambda_2\geq \lambda_1$, and $0\leq t_1,t_2\leq a$, we have
\[\bar{\sy}_\times^{n,[\lambda_2]}(t_2)-\bar{\sy}_\times^{n,[\lambda_1]}(t_1)\sim {\sf Bin}(N,q)\]
with 
\begin{align*}
N&= n^{5/3}(t_1\vee t_2-t_1\wedge t_2)-\frac{n^{4/5}}{2}\l((t_1\vee t_2)^2-(t_1\wedge t_2)^2\r)\\
q&= p_{\lambda_2}(n)-p_{\lambda_1}(n)=\frac{\lambda_2-\lambda_1}{n^{4/3}}.
\end{align*}
Again, a simple control of the fourth moment, the same as that of Section~\ref{sec:newproof_Aldous} 
suffices to conclude. 
The control of $\underline{\sy}^{n}_\times$ is done along the same lines. This suffices 
to conclude since the limits of $\bar{\sy}^{n,[\lambda^\star]}$ and $\underline{\sy}^{n,[\lambda^\star]}$ 
are both the same, as well as that of $\bar{\sy}^{n,[\lambda_\star]}$ and of $\underline{\sy}^{n,[\lambda_\star]}$, 
and since $\underline{\sy}^{n}\leq \sy \leq \bar{\sy}^{n}$.

\section{Standard coalescents: Proofs of Theorem~\ref{thm:conv-coal} and~\ref{thm:augmented_conv}}
\label{sec:proofs_sequences}

\subsection{Proof of Theorem~\ref{thm:conv-coal}}
\label{sec:proof_conv-coal}

Recall that a subset $K$ of $\ell^p$ has a compact closure if the following conditions hold:
\begin{enumerate}[$(i)$]
\item for any $i$ , $\{x_i ~: x_i \in K\}$ is bounded, and
\item for any $`e>0$, there exists $n\ge 1$ such that for all $x \in K$, $\sum_{k\geq n} |x_k|^p\leq `e$.
\end{enumerate}

Here, we deal with sequences of excursion-lengths that are non-negative by nature. 
So from now on, we focus on the subspace
$\ell^p_{\geq 0}=\l\{x=(x_i,i\geq 1), x_i\geq 0, \forall i, x \in \ell^p\r\}$. 
On $\ell^p$ define the operator $S^p$ by
\[S^p(x)=\Bigg(\sum_{j=1}^k x_j^p, k\ge 1\Bigg)\]
the map associating the $p$-partial sum of $x$. The map $S^p$ is a continuous injective 
on $\ell^p_{\geq 0}$. Write $K'=S^p(K)$. If $K$ satisfies $(i)$ and $(ii)$ then: 
\begin{enumerate}[$(i')$]
\item for any $k\ge 1$, $\{y_k: y\in K'\}$ is bounded;
\item for any $y\in K'$, the sequence $(y_k)_{k\ge 1}$ 
is non-decreasing an has a finite limit as $k\to +\infty$.
\end{enumerate}
Conversely, if $K'$ satisfies $(i')$ and $(ii')$ 
then $(S^p)^{-1}(K')$ has a compact closure.

We now come back to our model: $\lambda\to \bgam^{+,n}(\lambda)$ and 
$\lambda\to \bgam^{\times,n}(\lambda)$ that can be considered as elements of 
$\ell_{\geq 0}^1$ and $\ell_{\geq 0}^2$ respectively. An important property is that
for any $k$, any $n$, $\lambda\to S^1(\bgam^{+,n}(\lambda))_k$ and 
$\lambda\to S^2(\bgam^{\times,n}(\lambda))_k$ are both non decreasing 
(recall that the sequences $\bgam$ are sorted). 
The second point comes from the fact that $(x_i+x_j)^2\geq x_i^2+x_j^2$ implying that the 
coalescence of any two clusters will never decrease $S^2(x)_k$. 

\begin{lem}\label{lem:compact_coal}
Let  $I=[\lambda_\star,\lambda^\star]\subset `R$, and let $p\geq 0$. 
Let $\cK$ be a subset of  $\bbD(I,\ell^p_{\geq 0})$.
Assume that the following conditions hold:
\begin{compactenum}[$(a)$]
\item for any $f \in {\cal K}$, any integer $k\ge 1$, 
$\lambda\mapsto S^p(f(\lambda))_k$ is non-decreasing,
\item for any $k\ge 1$, $\{f(\lambda^\star)_k~:f \in {\cal K}\}$ is bounded, and
\item for any $\epsilon>0$, there exists $M\geq 0$ such that for any $m\ge M$ such that for any $f\in \cal K$ we have 
$$\inf_{\lambda\in I}S^p(f(\lambda))_m 
\ge \sup_{\lambda\in I} \lim_{k\to\infty} S^p(f(\lambda))_k -\epsilon.$$
\end{compactenum}
Then ${\cal K}$ is relatively compact.
\end{lem}
\begin{proof}Let $f_n$ be a sequence of functions in ${\cal K}$. One may extract from $f_n$ a 
subsequence $(f_{n[k]},k \geq 0)$ such that $\lambda\to S(f_n(\lambda))_1$ converges in $\bbD(I,`R)$ 
(any sequence of bounded non decreasing functions on $I$ has an accumulation point in $\bbD(I,`R)$). 
From a diagonal procedure one may further extract a subsequence $(f_{n[k]},k \geq 0)$ such that 
$\lambda\to S(f_n(\lambda))_k$ converges in $\bbD(I,`R^k)$ for any $k \leq K$ (for any $K$). 
Condition $(c)$ provides the necessary tightness and allows one to conclude. 
\end{proof}

Lemma~\ref{lem:compact_coal} permits to complete the proof of Theorem~\ref{thm:conv-coal}. 
We start with the additive case. Note that in this case, $\|\bgam^{+,n}(\lambda)\|_1=1$,
for every $\lambda$, so (b) is satisfied. Condition (a) also clearly holds. 
Furthermore, for every integer $m\ge 1$, the tail 
\[\|\bgam^{+,n}(\lambda)\|_1-S^1(\bgam^{+,n}(\lambda))_m=\sum_{i> m} \gamma^{+,n}(\lambda)\]
is largest at $\lambda=\lambda_\star$, and for Condition (c) in Lemma~\ref{lem:compact_coal} to 
be satisfied, it suffices that there exists $m$ such that for any $f\in {\cal K}$, we have 
\[S^1(f(\lambda_\star))_m \geq \lim_{k\to\infty} (S^1(f(\lambda_\star))_k)-`e.\]
Let $I=[\lambda_\star,\lambda^\star]\in(-\infty,0]$. 
By Theorem \ref{theo:main+},  $\sy^{n}_+ \to \sy_+$ in distribution in $\bbD(I,\bbC([0,1],`R))$,
which is separable. 
By Skorokhod representation theorem, there exists a probability space on which this
convergence holds a.s..
Let us work on this space and use the same names for the variables $\sy^{n}_+$ and $\sy_+$. 
The sum of excursion lengths above the minimum of $\sy_+(\lambda_\star)$ is 1; so for any $`e>0$, 
there exists a $k\ge 1$ such $S^1(\bgam^+(\lambda_\star))_k\geq 1-`e/2$. Hence, a.s.\ for all $n$ 
large enough, $S^1(\bgam^{n,+}(\lambda_\star))_k\geq 1-`e$, so that for any 
$\lambda\in[\lambda_\star,\lambda^\star]$, we have $S^1(\bgam^{n,+}(\lambda))_k\geq 1-`e$,
which proves that (c) holds. By Lemma~\ref{lem:compact_coal}, the family of functions 
$(\bgam^{+,n}(\lambda))_{\lambda \in [\lambda_\star,\lambda^\star]}$, $n\ge 1$, is 
tight in $\bbD([\lambda_\star,\lambda^\star], \ell^1_\da)$. 
Besides, on the space on which we work, a.s.\ for any 
$\lambda_\star\leq\lambda_1<\cdots \lambda_\ell\leq \lambda^\star$, any $1\leq j \leq \ell$,  
we have $\max_{i\leq k}|S(\bgam^{n,+})(\lambda_j)_i-S(\bgam^{+}(\lambda_j))_i|\to 0$, 
so that we have convergence in distribution.

In the multiplicative coalescent, conditions (a) and (b) easily follow from the monotonicity
and convergence in distribution of $\bgam^{\times,n}(\lambda^\star)$, as $n\to\infty$.  
Here, we are working in $\ell^2$ and since the norm is 
growing with $\lambda$ the tail of $\bgam^{+,n}(\lambda)$ is a priori not monotonic in $\lambda$. 
Thus, we cannot bound the supremum for $\lambda\in I$ by the value at $\lambda_\star$ as 
we did for the additive case and (c) is a little more delicate. 
However, for any $m\ge 1$, the only way for the tail 
\[\|\bgam^{\times,n}(\lambda)\|_2^2 - S^2(\bgam^{\times,n}(\lambda))_m
= \sum_{i>m} \gamma^{\times,n}(\lambda)^2\]
to decrease is that some mass 
gets moved to some component of index at most $m$ because of a coalescent that involves 
at least one component of index smaller that $m$. In other words
if we suppress the $m$ largest components, and let only the components of the tail evolve on 
their own according to the dynamics of the multiplicative coalescent, then their $2$-mass is monotonic (and their $\|.\|_2$ norm is larger than their contribution to the $\ell^2$ norm in the presence of 
the $m$ first ones). Hence, one may build on the same probability space a coupling of the initial coalescent and the one with the $m$ largest components suppressed. 

 For $m\ge 1$, define the $m$-tail to be the sequence
\begin{equation}\label{eq:def_tail}
\btau^{[m],n}(\lambda):=(\gamma^{\times,n}_{m+1}(\lambda),\gamma^{\times,n}_{m+2}(\lambda),\dots) 
\in \ell^2_\da.
\end{equation}
 Let $\bT^{[m],n}(\lambda)$ denote the state of 
the multiplicative coalescent at time $\lambda$ when started from 
$\bT^{[m],n}(\lambda_\star)=\btau^{[m],n}(\lambda_\star)$.
Then, by the previous discussion we have (on a space) for $\lambda \in [\lambda_\star,\lambda^\star]$, 
\begin{equation}\label{eq:bound_tail_tau}
\|\btau^{[m],n}(\lambda)\|_2^2 \le \|\bT^{[m],n}(\lambda)\|_2^2.
\end{equation}

Now, in order to prove that $(c)$ in Lemma~\ref{lem:compact_coal} holds, we rely on the
Feller property. First note that, by the triangle inequality, for every $m,n\ge 1$,
\[
\|\btau^{[m],n}(\lambda_\star)\|_2
\le \|\btau^{[m]}(\lambda_\star)\|_2 + \|\bgam^{\times,n}(\lambda_\star)-\bgam^{\times}(\lambda_\star)\|_2.
\]
So the starting point of $\bT^{[m],n}$ can be made arbitrarily small by choice of $m$ and $n$ 
large enough, with high probability: $\|\btau^{[m],n}(\lambda_\star)\|_2\to 0$ in distribution 
as $n,m\to \infty$. Thus, by the (spatial) Feller property, we obtain at time $\lambda$
\[\|\bT^{[m],n}(\lambda^\star)\|_2\to 0,\]
in distribution as $m,n\to\infty$. Together with \eqref{eq:bound_tail_tau}, this yields the desired 
uniform bound on $\btau^{[m],n}$ on $[\lambda_\star,\lambda^\star]$. 
It follows that (c) is satisfied with high probability, which 
is sufficient to prove tightness of the family of functions $(\bgam^{\times,n}(\lambda))_{\lambda\in I}$,
$n\ge 1$, in $\bbD(I, \ell^2_\da)$. To complete the proof of the convergence, 
it suffices now to prove convergence of the fdd.

To this aim, use Theorem~\ref{theo:mainx} and a Skorokhod's representation in order to obtain 
a sequence of functions $(\sy_\times^n)_{n\ge 1}$ that converge a.s.\ in 
$\bbD(\bbR^+\times \bbR, \bbR)$. Then, in particular, for any $\ell\ge 1$ and for any 
$\lambda_\star\le \lambda_1<\lambda_2<\dots<\lambda_\ell\le \lambda^\star$, we have
\[\big(\sy_\times^{n,(\lambda_1)},\dots, \sy_\times^{n,(\lambda_\ell)}\big)
\to \big(\sy_\times^{(\lambda_1)},\dots, \sy_\times^{(\lambda_\ell)}\big)\] 
a.s.\ as $n\to\infty$.
The result follows from a simple adaptation of the arguments leading to Lemma~7 
of \cite[Section 2.3]{aldous1997}. 

\begin{proof}[Proof of Corollary~\ref{thm:main_continuous}]
Since the finite-dimensional distributions (fdd) characterize the law of the process,
it suffices to verify that the fdd of $(\bgam^+(e^{-t}))_{t\in \bbR}$ and 
$(\bgam^\times(\lambda)_{\lambda\in \bbR}$ coincide with those of the standard additive and 
multiplicative coalescents, respectively. 

Consider any natural number $k\ge 1$ and any $t_1<t_2<\dots<t_k$. Then, for fixed $n\ge 1$, 
the vector $(\bgam^{+,n}(e^{-t_1}), \dots, \bgam^{+,n}(e^{-t_k}))$ is 
distributed as the values at times $t_1,t_2,\dots, t_k$ of the additive coalescent started at time $t_1$ 
in the state $\bgam^{+,n}(e^{-t_1})$. Since 
$$(\bgam^{+,n}(e^{-t_1}), \dots, \bgam^{+,n}(e^{-t_k})) \to (\bgam^{+}(e^{-t_1}), \dots, \bgam^{+}(e^{-t_k}))$$
in distribution as $n\to\infty$, the Feller property of the additive coalescent implies that if the coalescent 
is started at time $t_1$ in state $\bgam^{+}(e^{-t_1})$, then the distributions at times $t_1,t_2,\dots, t_k$ are 
given by 
$$(\bgam^{+}(e^{-t_1}), \dots, \bgam^{+}(e^{-t_k})).$$ 
In other words, the fdds of $(\bgam^+(e^{-t}))_{t\in `R}$
are those of the an additive coalescent, so that it is in fact an additive coalescent. 
We have the \emph{standard} additive coalescent since for a single fixed time $t$, $\bgam^+{e^{-t}}$ 
is the scaling limit of a percolated Cayley tree. The proof in the multiplicative case is similar, 
one only needs to use Note~\ref{note:homogeneous} to make the process time-homogeneous for fixed $n$, 
and identify the \emph{standard} multiplicative coalescent because $\bgam^\times(\lambda)$ is 
the scaling limit of the cluster sizes of $G(n,p_\lambda(n))$; the details are omitted.
\end{proof}

\subsection{Proof of Theorem~\ref{thm:augmented_conv}}

The representation we have of the process also yields a construction of 
a standard version of the augmented multiplicative coalescent of \citet{BhBuWa2013b}, 
which also keeps tracks of the surplus of the connected components (the minimum number of 
edges to remove in order to obtain a tree). We start with Corollary~\ref{cor:dist_incr}. 
Consider the random variables $U_i^\star(k)$, $i<k\le n$, and arrange them geometrically in a field 
$(P^n_i(j))_{1\le i \le n, 1\le j\le n-i}$ on the first quadrant $`N\times `N$ by setting
$P_i^n(j):=U^\star_{i+1}(i+1+j)$ for $1\le i\le n$ and $1\le j \le k-i$. 
Then, the number of extra edges in $G(n,t)$ of a connected component corresponding to an interval $I$ is
precisely the number of $P^n_i(j)$, $i\in I$, lying below the graph of $Z^t$ whose value is at most $t$.
(See Figure~\ref{fig:augmented}.)

\begin{proof}[Proof of Theorem~\ref{thm:augmented_conv}]
For every $t=1/n+\lambda n^{-4/3}$, the Bernoulli point set $\{(i n^{-2/3}, jn^{-1/3}): P^n_i(j)\le t\}$ 
converges to a Poisson point process with intensity one on $`R^+\times `R^+$. Furthermore, the limit 
point set is independent of $\lambda\in `R$. 

The convergence of the finite-dimensional distributions follows from arguments similar to the 
ones used in the proof of Theorem~\ref{thm:conv-coal} which rely on Skorokhod's representation theorem.
So it suffices to prove convergence of the marginals. For fixed $t$, note that the random
variables of the Bernoulli point set that are used to define $Z^t$ and the ones 
used to define the surplus are distinct. It follows that, for fixed $\lambda$, and on a space on which
$\Psi y^{(\lambda),n}_\times$ converges a.s., we have 
\[((\gamma_i^{\times,n}(\lambda))_{1\le i\le k}, (s_i^n(\lambda))_{1\le i\le k})
\to ((\gamma_i^{\times}(\lambda))_{1\le i\le k},(s_i(\lambda))_{1\le i\le k}),
\]
for every integer $k\ge 1$. It now remains to prove that $(\bgam_i^{\times,n}(\lambda),\bs^n(\lambda))_{n\ge 1}$ 
is tight in $\bbU_\da$. Since $\bgam^{\times,n}(\lambda)$ converges in $\ell^2$, it suffices to 
verify that for any $\epsilon>0$, there exists $K$ such that
\[\sup_{n\ge 1} \p{\sum_{i\ge 1} \gamma_i^{\times,n}(\lambda) \cdot s_i^n(\lambda) \ge K }\le \epsilon.\]
Let $I_i(\lambda)$ denote the interval corresponding to
the excursion of $\Psi \sy^{(\lambda),n}_\times$ whose length is recorded in $\gamma_i^{\times,n}(\lambda)$.
Let also $a_i^n$ denote the number of integral points lying (strictly) between the horizontal axis and 
the graph of $\Psi \sy^{(\lambda),n}_\times$ on the interval $I_i(\lambda)$. Then,
\[\Ec{\gamma_i^{\times,n}(\lambda)\cdot s_i^n(\lambda)} 
= p_\lambda(n) \cdot \Ec{\gamma_i^{\times,n}(\lambda) \cdot a_i^n(\lambda)}.\]
However $a_i^n$ is the area of the tree associated to the interval $I_i(\lambda)$ in the sense 
of \cite[Section~2]{AdBrGo2012a}. In particular, given $\gamma_i^{\times,n}(\lambda)n^{2/3}=m$, $a_i(\lambda)$ 
is distributed as $a(\tilde T^p_m)$ there. More precisely, 
\[
\pc{\tilde T^p_m \in \cB} 
= \frac{\Ec{\I{T_m\in \cB} (1-p_\lambda(n))^{-a(T_m)}}}{\Ec{(1-p_\lambda(n))^{-a(T_m)}}},
\]
where $T_m$ is a uniformly random tree on $[m]$, but we shall only use that 
$a(\tilde T^p_m) \le m h(\tilde T^p_m)$, where 
$h(T)$ denotes the height of the labelled tree $T$, and refer to some bounds on the height of 
$\tilde T_m^p$ proved in \cite{AdBrGo2012a}. 
By Lemma~25 there and Jensen's inequality, there is a constant $C$ such that, 
for all $1\le m\le n$, we have
\begin{equation}\label{eq:unif_bound_area}
\Ec{a(\tilde T_m^p)}\le m \Ec{h(\tilde T_m^p)} \le C \max\{m^{3/2}/n, 1\} \cdot m^{3/2}.
\end{equation}
Note that here, $m$ is $n^{2/3}\gamma_i^{\times,n}$ for some $i\ge 1$, and 
$m^{3/2}/n\le (\gamma_1^{\times,n})^{3/2}\le \|\bgam^{\times,n}\|^{3/2}_2$. 
Since for $n$ large enough we have $p_\lambda(n)\le 2/n$, it follows that
\begin{align*}
`E\Bigg[\sum_{i\ge 1} \gamma_i^{\times,n}(\lambda)\cdot s_i^n(\lambda)
\,\Bigg| \,\|\bgam^{\times,n}(\lambda)\|_2\Bigg]
& = p_\lambda(n) `E\Bigg[\sum_{i\ge 1} \gamma_i^{\times,n}(\lambda)\cdot a_i^n(\lambda)
\,\Bigg| \,\|\bgam^{\times,n}(\lambda)\|_2\Bigg]\\
& \le 2C \E{\max\{\|\bgam^{\times,n}(\lambda)\|^{3/2}_2, 1\} \cdot \|\bgam^{\times,n}(\lambda)\|_{5/2}^{5/2}
\, \Big|\,\|\bgam^{\times,n}(\lambda)\|_2} \\
& \le 2C \max\{\|\bgam^{\times,n}(\lambda)\|_2^{4}, \|\bgam^{\times,n}(\lambda)\|_2^{5/2}\}.
\end{align*}
Since $(\bgam^{\times,n}(\lambda))_{n\ge 1}$ is tight in $\ell^2$, it follows that 
$(\bgam^{\times,n}(\lambda), \bs^{n}(\lambda))_{n\ge 1}$ is tight in $\bbU_\da$. 

The tightness of the sequence of processes 
$(\bgam^{\times,n}(\lambda), \bs^n(\lambda))_{\lambda\in [\lambda_\star,\lambda^\star] }$, $n\ge 1$,
in $\bbD([\lambda_\star,\lambda^\star],\bbU_\da)$ follows from the almost Feller property of the augmented 
multiplicative coalescent \cite{BhBuWa2013b}, as in the proof of Theorem~\ref{thm:conv-coal}.
It suffices to consider the modified cumulative operator $\Sigma: \bbU_\da\to \ell^2_{\ge 0}$ 
that associates to $(\bx,\bs)\in \bbU_\da$ the sequence
\[\Sigma(\bx,\bs):=\l(\sum_{i=1}^k x_i^2 + \sum_{i=1}^k x_i s_i, k\ge 1\r).\]
Then relative compactness in $\bbD([\lambda_\star,\lambda^\star], \bbU_\da)$ reduces to the 
three conditions in Lemma~\ref{lem:compact_coal} with $S^p$ replaced by $\Sigma$. The proof 
uses no idea that is not already present in the proof of Theorem~\ref{thm:conv-coal} that 
we detailed earlier, so we omit the details.
\end{proof}

\begin{figure}[htb]
    \centering
    \includegraphics[scale=.7]{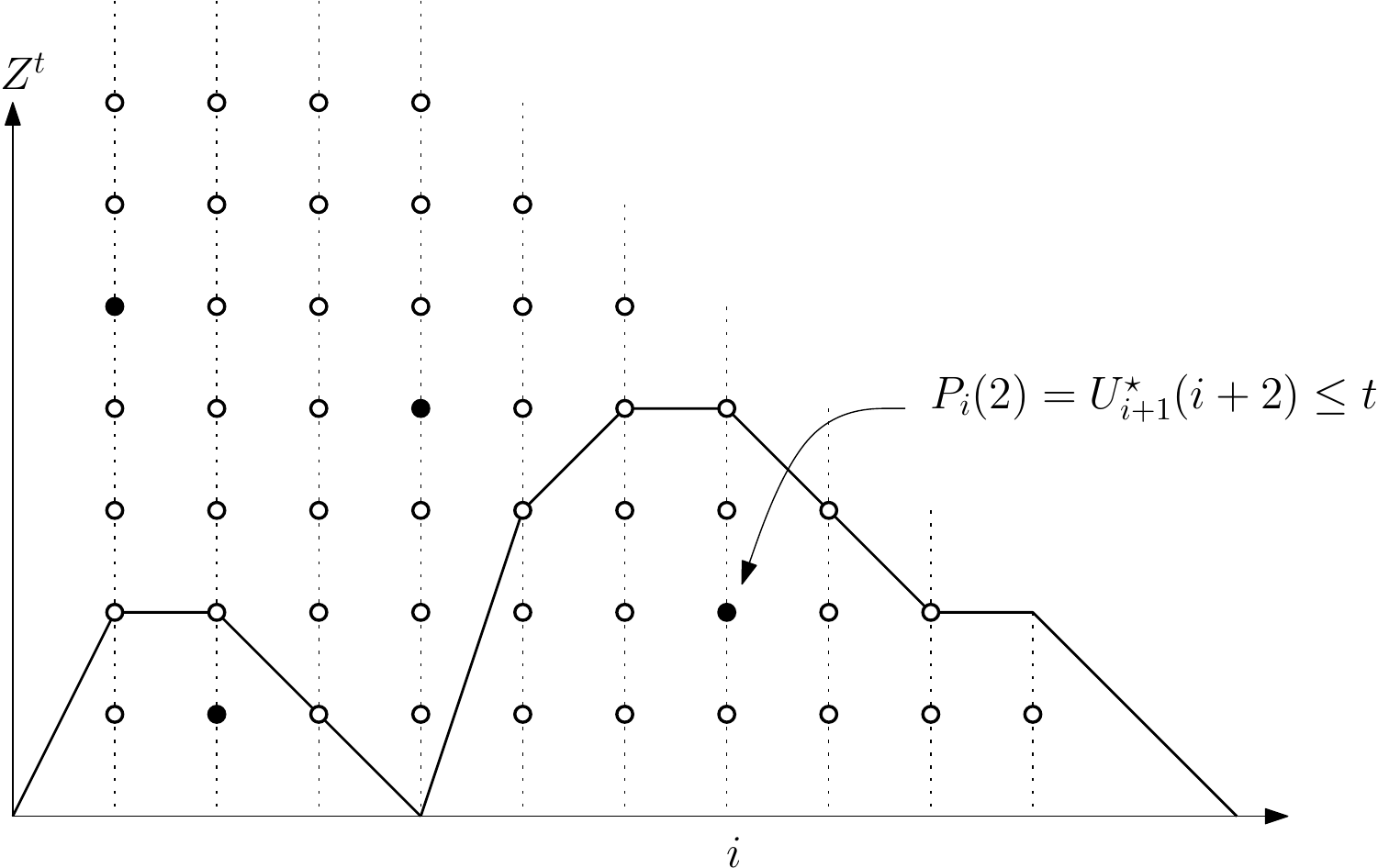}
    \caption{\label{fig:augmented} $Z^t$ and the field $(P^n_i(j))$ is represented by the bullets;
    the black ones are the ones whose value is at most $t$. So here, the graph represented has 
    two connected components, each having one extra edge.}
\end{figure}

{\small
\setlength{\bibsep}{0.25em}
\bibliographystyle{plainnat}
\bibliography{bib}
}

\end{document}